\documentclass[a4paper,10pt]{article}
\usepackage{amsmath,amssymb,amsfonts,amsthm}
\usepackage{graphicx,color,subfigure,enumerate,multirow,accents}
\usepackage[]{hyperref}
\usepackage{a4wide}
\usepackage[dvipsnames]{xcolor}
\usepackage[a4paper,width={15cm},left=3cm,bottom=3.5cm, top=3.5cm]{geometry}

\newcommand{\R}{\mathbb R}
\newcommand{\C}{\mathbb C}
\newcommand{\N}{\mathbb N}

\newcommand{\eps}{\varepsilon}
\newcommand{\abs}[1]{\left\vert #1 \right\vert}

\newcommand{\Div}[1]{\mathrm{div}#1}
\newcommand{\Di}{\mathcal{D}^{1,2}}
\renewcommand{\Re}{\mathop{\mathfrak{Re}}}
\renewcommand{\Im}{\mathop{\mathfrak{Im}}}

\newtheorem{theorem}{Theorem}[section]
\newtheorem{proposition}[theorem]{Proposition}

\newtheorem{lemma}[theorem]{Lemma}

\newtheorem{remark}[theorem]{Remark}

\makeatletter
\let\@fnsymbol\@alph
\makeatother

\begin{document}

\title{Rate of convergence
  for eigenfunctions of Aharonov-Bohm operators
  with a moving pole}
\author{
L. Abatangelo\footnote{
Dipartimento di Matematica e Applicazioni,
 Universit\`a di Milano--Bicocca,
Via Cozzi 55, 20125 Milano, Italy,
\texttt{laura.abatangelo@unimib.it}}, V. Felli\footnote{
Dipartimento di Scienza dei Materiali,
 Universit\`a di Milano--Bicocca,
Via Cozzi 55, 20125 Milano, Italy,
\texttt{veronica.felli@unimib.it}}}

\maketitle

\begin{abstract}
We study the behavior of eigenfunctions for magnetic
  Aharonov-Bohm operators with half-integer circulation and Dirichlet
  boundary conditions in a planar domain. 
We prove a sharp estimate for the rate of convergence of
eigenfunctions as the pole moves in the interior of
  the domain.
\end{abstract}



\section{Introduction}

For every $a=(a_1,a_2)\in\R^2$, we consider the Aharonov-Bohm vector potential
with pole $a$ and circulation $1/2$  defined as 
\[
A_a(x_1,x_2)=A_0(x_1-a_1,x_2-a_2),\quad 
(x_1,x_2)\in\R^2\setminus\{a\},
\]
where
\[
A_0(x_1,x_2)=\frac12\bigg(\frac{-x_2}{x_1^2+x_2^2},
\frac{x_1}{x_1^2+x_2^2}\bigg),\quad 
(x_1,x_2)\in\R^2\setminus\{(0,0)\}.
\]
The Aharonov-Bohm vector potential $A_a$ generates a $\delta$-type magnetic
field, which is called Aharo\-nov--Bohm field: this field is produced by 
an infinitely long thin solenoid intersecting perpendicularly the plane $(x_1, x_2)$ at the point $a$,
as the radius of the solenoid tends to zero while the flux through the
solenoid section remains constantly equal to $1/2$.  Negletting the
irrelevant coordinate
 along the solenoid axis, the problem becomes $2$-dimensional.

Let  $\Omega\subset\R^2$ be a bounded, open and simply connected
domain.  
For every $a\in \Omega$,  we consider the eigenvalue
problem 
\begin{equation}\label{eq:eige_equation_a}\tag{$E_a$}
  \begin{cases}
   (i\nabla + A_{a})^2 u = \lambda u,  &\text{in }\Omega,\\
   u = 0, &\text{on }\partial \Omega,
 \end{cases}
\end{equation}
in a weak sense, 
where the  magnetic Schr\"odinger operator with Aharonov-Bohm vector potential
$(i\nabla +A_{a})^2$ acts on functions 
$u:\R^2\to\C$ as
\[
(i\nabla +A_{a})^2 u=-\Delta u+2iA_{a}\cdot\nabla u+|A_{a}|^2 u.
\]
 A suitable functional setting
for stating a weak formulation of \eqref{eq:eige_equation_a} can be
introduced as follows: for every $a\in\Omega$, the functional space
$H^{1 ,a}(\Omega,\C)$ is defined as the completion of
\[
\{u\in
H^1(\Omega,\C)\cap C^\infty(\Omega,\C):u\text{ vanishes in a
  neighborhood of }a\}
\]
 with respect to the norm 
\begin{equation*}
 \|u\|_{H^{1,a}(\Omega,\C)}=  \left(\left\|(i\nabla+A_{a}) u\right\|^2
    _{L^2(\Omega,\C^2)} +\|u\|^2_{L^2(\Omega,\C)}\right)^{\!\!1/2}.
\end{equation*}
In view of the following
Hardy type inequality proved in \cite{LW99} 
\begin{equation*}
  \int_{\R^2} |(i\nabla+A_a)u|^2\,dx \geq \frac14 \int_{\R^2}\frac{|u(x)|^2}{|x-a|^2}\,dx,
\end{equation*}
which holds for all $a\in \R^2$  and $u\in C^\infty_{\rm
  c}(\R^2\setminus\{0\},\C)$, it is easy to verify that 
\[
H^{1,a}(\Omega,\C)=\big\{u\in H^1(\Omega,\C):\tfrac{u}{|x-a|}\in
  L^2(\Omega,\C)\big\}.
\] 
We also denote as $H^{1 ,a}_{0}(\Omega,\C)$ the space obtained as the completion
of $C^\infty_{\rm c}(\Omega\setminus\{a\},\C)$ with respect to the
norm $\|\cdot\|_{H^{1,a}(\Omega,\C)}$, so that 
$H^{1,a}_{0}(\Omega,\C)=\big\{u\in H^1_0(\Omega,\C):\frac{u}{|x-a|}\in
L^2(\Omega,\C)\big\}$.

For every $a\in\Omega$,  we say that $\lambda\in\R$ is an eigenvalue
of problem \eqref{eq:eige_equation_a}
in a weak sense if there exists $u\in
H^{1,a}_{0}(\Omega,\C)\setminus\{0\}$ (called an eigenfunction) such that
\begin{equation*}
\int_\Omega (i\nabla u+A_{a} u)\cdot \overline{(i\nabla v+A_{a}
  v)}\,dx=\lambda\int_\Omega u\overline{ v}\,dx \quad\text{for all }v\in H^{1,a}_{0}(\Omega,\C).
\end{equation*}
From classical spectral
theory, the eigenvalue problem $(E_a)$ admits a sequence of real
diverging eigenvalues (repeated according to
their finite multiplicity) $\lambda_1^a \leq \lambda_2^a\leq\dots\leq
\lambda_j^a\leq\dots$.

The mathematical interest in Aharonov-Bohm operators with half-integer circulation
can be  motivated by a strong relation between
spectral
minimal partitions of the Dirichlet Laplacian with points of odd
multiplicity and nodal domains of
eigenfunctions of these operators.
Indeed, a magnetic characterization of minimal partitions was given
in \cite{HHO13} (see also \cite{BNH2011,BNHHO2009,BNHV2010,NT}):  partitions with  points of
odd multiplicity can be obtained as nodal domains by
minimizing a certain eigenvalue of an Aharonov-Bohm Hamiltonian with
respect to the number and the position of poles. From this, a natural
interest in the study of the properties of
the map $a\mapsto\lambda_j^a$ (associating eigenvalues of magnetic operators to the position
of poles) arises.
In \cite{AF, AF2, AFNN, BNNNT, lena, NNT} the behaviour of the function
$a\mapsto \lambda_j^a$ in a neighborhood of a fixed point $b\in
\overline{\Omega}$ has been investigated, both in the cases
$b\in\Omega$ and $b\in\partial\Omega$.
In particular, the analysis carried out in \cite{AF, AF2, AFNN, BNNNT,
  NNT}  shows that, as the pole moves towards a fixed limit pole $b\in
\overline{\Omega}$,  
the rate of  convergence of $\lambda_j^a$ to $\lambda_j^b$  is related 
to the number of nodal
lines of the limit eigenfunction meeting at $b$.
 In the present paper we aim at deepening this analysis 
describing also the behaviour of the corresponding eigenfunctions; in
particular, we will derive a sharp estimate for the rate of
convergence of eigenfunctions associated to moving poles, in terms
of the number of nodal lines of the limit eigenfunction. 

Without loss of generality, we can assume that 
\[
b=0\in\Omega.
\] 
Let us assume that there exists $n_0\geq 1$ such that 
\begin{equation}\label{eq:1}
  \lambda_{n_0}^0\quad\text{is simple},
\end{equation}
and denote $\lambda_0= \lambda_{n_0}^0$ 
and, for any $a\in\Omega$,
$\lambda_a= \lambda_{n_0}^a$.
From \cite[Theorem 1.3]{lena} it follows that the map  $a\mapsto
\lambda_a$ is analytic in a neighborhood of 0; in particular we have
that 
\begin{equation}\label{eq:conv_auto}
  \lambda_a\to \lambda_0, \quad\text{as }a\to0. 
\end{equation}
Let $\varphi_0\in H^{1,0}_{0}(\Omega,\C)\setminus\{0\}$ be a
$L^2(\Omega,\C)$-normalized
eigenfunction of problem $(E_0)$ associated to the eigenvalue
$\lambda_0= \lambda_{n_0}^0$, i.e. satisfying
\begin{equation}\label{eq:equation_lambda0}
 \begin{cases}
   (i\nabla + A_0)^2 \varphi_0 = \lambda_0 \varphi_0,  &\text{in }\Omega,\\
   \varphi_0 = 0, &\text{on }\partial \Omega,\\
\int_\Omega |\varphi_0(x)|^2\,dx=1.
 \end{cases}
\end{equation}
From \cite[Theorem 1.3]{FFT} (see also  \cite[Theorem 1.5]{NT}) it is known that
  $\varphi_0$ has at $0$ a zero
    of order $\frac k2$ for some odd $k\in \N$,
i.e. there exist $k\in\N$ odd and $\beta_1,\beta_2\in\C$ such that
$(\beta_1,\beta_2)\neq(0,0)$ and
\begin{equation}\label{eq:131}
  r^{-k/2} \varphi_0(r(\cos t,\sin t)) \to 
  e^{i\frac t2}\left(\beta_1
 \cos\Big(\frac k2
  t\Big)+\beta_2 
  \sin\Big(\frac k2
  t\Big)\right) \quad \text{in }C^{1,\tau}([0,2\pi],\C)
\end{equation}
as $r\to0^+$ for any $\tau\in (0,1)$.
The asymptotics \eqref{eq:131} (together with the fact that the
right hand side of \eqref{eq:131} is a complex multiple of a
real-valued function, see \cite{HHOO99})
implies that $\varphi_0$ has exactly $k$ nodal lines meeting
at $0$ and dividing the whole angle into $k$ equal parts; 
such  nodal
lines are tangent to the $k$ half-lines 
\[
\bigg\{\bigg(t,\tan
\Big(\alpha_0+j\frac{2\pi}k\Big) t \bigg):\,t>0 \bigg\}, \quad j=0,1,\dots,k-1,
\]
for some angle $\alpha_0\in [0,\frac{2\pi}{k})$.

In \cite{AF,AF2} it has been proved that, under assumption
(\ref{eq:1}) and being $k$ as in \eqref{eq:131}, 
\begin{equation}\label{eq:4}
\frac{\lambda_0-\lambda_a}{|a|^k}\to C_0
\cos\big(k(\alpha-\alpha_0)\big) 
 \qquad \text{as $a\to0$ with $a=|a|(\cos\alpha,\sin\alpha)$},
\end{equation}
where $C_0>0$ is a positive constant depending only on $k$, $\beta_1$,
and $\beta_2$. More precisely, in \cite{AF,AF2} it has been proved
that 
\begin{equation*}
C_0=
-4(|\beta_1|^2+|\beta_2|^2)\,{\mathfrak m}_k
\end{equation*} 
where 
\begin{equation}\label{eq:12}
{\mathfrak m}_k=\min_{u\in \Di_{s}(\R^2_+)}\left[\frac12 \int_{\R^2_+} |\nabla u(x)|^2 \,dx-
 \frac k2\int_0^1 t^{\frac k2-1} u(t,0)\,dt\right]<0.
\end{equation}
In (\ref{eq:12}), $s$ denotes the half-line $s:=\{(x_1,x_2)\in\R^2: x_2=0\text{ and }x_1\geq
1\}$ and $\Di_{s}(\R^2_+)$ is 
the completion of $C^\infty_{\rm c}(\overline{\R^2_+} \setminus s)$
under the norm $( \int_{\R^2_+} |\nabla u|^2\,dx )^{1/2}$.

Let us now consider a suitable family of eigenfunctions relative to
the approximating eigenvalue $\lambda_a$. In order to choose
eigenfunctions with a suitably normalized phase, let us introduce the
following notations. For every $\alpha\in[0,2\pi)$  and 
$b=(b_1,b_2)=|b|(\cos\alpha,\sin\alpha)\in \R^2\setminus\{0\}$, we define
\begin{equation*}
\theta_b:\R^2\setminus\{b\}\to [\alpha,\alpha+2\pi)
\quad\text{and}\quad
\theta_0^b:\R^2\setminus\{0\}\to [\alpha,\alpha+2\pi)
\end{equation*}
such that 
\[
\theta_b(b+r(\cos t,\sin t))=t
\quad \text{and}\quad \theta_0^b(r(\cos t,\sin t))=t,
\quad\text{for all }r>0\text{ and }t\in  [\alpha,\alpha+2\pi). 
\]
We also define
\[
\theta_0:\R^2\setminus\{0\}\to [0,2\pi)\]
such that 
\[
\theta_0(\cos t,\sin t)=t\quad\text{for all
}t\in[0,2\pi).
\]
For all $a\in\Omega$, let $\varphi_a\in
H^{1,a}_{0}(\Omega,\C)\setminus\{0\}$ be an eigenfunction of problem
\eqref{eq:eige_equation_a} associated to the eigenvalue $\lambda_a$,
i.e. solving
\begin{equation}\label{eq:equation_a}
 \begin{cases}
   (i\nabla + A_a)^2 \varphi_a = \lambda_a \varphi_a,  &\text{in }\Omega,\\
   \varphi_a = 0, &\text{on }\partial \Omega,
 \end{cases}
\end{equation}
such that its modulus and phase are normalized in such a way that  
\begin{equation}\label{eq:6}
  \int_\Omega |\varphi_a(x)|^2\,dx=1 \quad\text{and}\quad 
  \int_\Omega e^{\frac i2(\theta_0^a-\theta_a)(x)}\varphi_a(x)\overline{\varphi_0(x)}\,dx\text{ is a
    positive real number},
\end{equation}
where $\varphi_0$ is as in \eqref{eq:equation_lambda0}.
From \eqref{eq:1}, \eqref{eq:conv_auto}, \eqref{eq:equation_lambda0},
\eqref{eq:equation_a}, \eqref{eq:6}, and standard elliptic estimates,
it follows that $\varphi_a\to \varphi_0$ in $H^1(\Omega,\C)$ and in
$C^2_{\rm loc}(\Omega\setminus\{0\},\C)$
and 
\begin{equation}\label{eq:congrad2}
(i\nabla+A_a)\varphi_a\to (i\nabla+A_0)\varphi_0 \quad\text{in }L^2(\Omega,\C).
\end{equation}
The main result of the present paper establishes the sharp rate of the
convergence \eqref{eq:congrad2}.
\begin{theorem}\label{t:main}
  For $\alpha\in\R$, $p=(\cos\alpha,\sin\alpha)$ and
  $a=|a|p\in\Omega$, let $\varphi_a\in
  H^{1,a}_{0}(\Omega,\C)$ solve equation (\ref{eq:equation_a}-\ref{eq:6}) and
  $\varphi_0\in H^{1,0}_{0}(\Omega,\C)$ be a solution to
  \eqref{eq:equation_lambda0} satisfying \eqref{eq:1} and
  \eqref{eq:131}.  Then there exists ${\mathfrak L}_p>0$ such that 
\begin{equation}\label{eq:13}
|a|^{-k}\left\|(i\nabla+A_a)\varphi_a-e^{\frac{i}{2}(\theta_a-\theta_0^a)}(i\nabla+A_0)\varphi_0\right\|^2_{L^2(\Omega,\C)}
\to (|\beta_1|^2+|\beta_2|^2){\mathfrak L}_p
\end{equation}
as $a=|a|p\to0$. Moreover the function $\alpha\mapsto {\mathfrak
  L}_{(\cos\alpha,\sin\alpha)}$ is continuous, even, and periodic with period $\frac{2\pi}k$. 
\end{theorem}
The constant ${\mathfrak L}_p$ in Theorem \ref{t:main} can be
characterized as the energy of the solution of an elliptic problem
with cracks (see \eqref{eq:wp}), where jumping conditions are prescribed on the segment
connecting $0$ and $p$ and on the tangent to a nodal line of
$\varphi_0$, see section \ref{sec:vari-char-limit}.

For every $\alpha\in \R$, let us denote as
$s_\alpha=\{t(\cos\alpha,\sin\alpha):t\geq0\}$ the half-line with
slope $\alpha$. We  notice that, if $a=|a|(\cos\alpha,\sin\alpha)$,
then $\nabla\big(\frac{\theta_a}{2}\big)=A_a$, 
$\nabla\big(\frac{\theta_0^a}{2}\big)=A_0$, and 
$e^{-\frac{i}{2}\theta_a}$ and $e^{-\frac{i}{2}\theta_0^a}$ are
smooth in $\Omega\setminus s_\alpha$. Thus
\[
i\nabla_{\Omega\setminus
  s_\alpha}(e^{-\frac{i}{2}\theta_a}\varphi_a)=e^{-\frac{i}{2}\theta_a}(i\nabla+A_a)\varphi_a,\quad 
i\nabla_{\Omega\setminus
  s_\alpha}(e^{-\frac{i}{2}\theta_0^a}\varphi_0)=e^{-\frac{i}{2}\theta_0^a}(i\nabla+A_0)\varphi_0,
\] 
where $\nabla_{\Omega\setminus s_\alpha}$ is the distributional
gradient in $\Omega\setminus s_\alpha$. Hence \eqref{eq:13} can be
rewritten as 
\[
|a|^{-k}\left\|
\nabla_{\Omega\setminus
  s_\alpha}(e^{-\frac{i}{2}\theta_a}\varphi_a-e^{-\frac{i}{2}\theta_0^a}\varphi_0)\right\|^2_{L^2(\Omega,\C)}
\to (|\beta_1|^2+|\beta_2|^2){\mathfrak L}_p
\]
as $a=|a|p\to0$; thus it can be interpreted as a sharp asymptotics of
the rate of convergence of the approximating eigenfunction
to the limit eigenfunction in the space $\{u\in H^1(\Omega\setminus
s_\alpha):u=0\text{ on }\partial\Omega\}$.

The paper is organized as follows. In section \ref{sec:preliminaries}
we fix some notation and recall some known facts. In section
\ref{sec:vari-char-limit} we give a variational characterization of
the limit profile of scaled eigenfunctions, which is used to study the
properties (positivity, evenness, periodicity) of the function
$p\mapsto {\mathfrak L}_p$. Finally, in section
\ref{sec:rate-conv-eigenf} we prove Theorem \ref{t:main}, 
providing  estimates of the energy variation first inside disks with radius
$R|a|$ and then outside such disks; this latter outer estimate is
performed exploiting the
invertibility of an operator associated to the limit eigenvalue
problem. 
We mention that this strategy was first developed in \cite{AFT} in the
context of spectral stability for varying domains, obtained by adding thin handles to a fixed
limit domain.

\section{Preliminaries and some known facts}\label{sec:preliminaries}

Through a rotation, we can easily choose a coordinate system in such a way that one 
nodal line of $\varphi_0$ is tangent to the $x_1$-axis,
i.e. $\alpha_0=0$. In this coordinate system, we have that, letting
$\beta_1,\beta_2$ be as in   
\eqref{eq:131}, 
 \begin{equation}\label{eq:54}
  \beta_1=0.
\end{equation}
The asymptotics of eigenvalues established in \cite{AF,AF2}, as well
as the
estimates for eigenfunctions we are going  to achieve in the present
paper, are based on a blow-up analysis for scaled eigenfunctions
performed in \cite{AF,AF2},
whose main results are briefly recalled below for the sake of completeness.

For every $p\in\R^2$ and $r>0$, we denote as 
$D_r(p)$
 the disk of center $p$ and radius $r$ and  as 
$D_r=D_r(0)$
 the disk of center $0$ and radius $r$. Moreover we denote, for every
 $r>0$, $D_r^+=\{(x_1,x_2)\in D_r:x_2>0\}$ and 
$D_r^-=\{(x_1,x_2)\in D_r:x_2<0\}$.

First of all, we observe that \eqref{eq:131} completely describes the
behaviour of $\varphi_0$ after scaling; indeed, letting 
\begin{equation*}
 W_a(x):=\frac{\varphi_0(|a|x)}{|a|^{k/2}}, 
\end{equation*}
 from \cite[Theorem 1.3 and
Lemma 6.1]{FFT}  we have  that,  under condition \eqref{eq:54},
\begin{equation}\label{eq:vkext_la}
W_a\to  \beta_2 e^{\frac i2\theta_0}\psi\quad\text{as } |a|\to0
\end{equation}
in $H^{1 ,0}(D_R,\C)$ for
every $R>1$, where $\psi:\R^2\to\R$ is the $\frac k2$-homogeneous function
(which is harmonic on $\R^2\setminus\{(r,0):r\geq0\}$)
\begin{equation}\label{eq:2}
 \psi(r\cos t,r\sin t)= r^{k/2} \sin
  \bigg(\frac{k}{2}\,t\bigg),\quad 
  r\geq0,\quad t\in[0,2\pi].
\end{equation}
For every ${p}\in\R^2$, we denote by ${\mathcal
  D}^{1,2}_{  p}(\R^2,\C)$ the completion of $C^\infty_{\rm
  c}(\R^N\setminus\{0\},\C)$ with respect to the magnetic Dirichlet
norm
\begin{equation}\label{eq:D12p}
  \|u\|_{{\mathcal D}^{1,2}_{ 
      p}(\R^2,\C)}:=\bigg(\int_{\R^2}\big|(i\nabla +A_{  p})u(x)\big|^2\,dx\bigg)^{\!\!1/2}.
\end{equation}

\begin{proposition}[\cite{AF2}, Proposition 4]\label{prop_Psi}
Let $\alpha\in[0,2\pi)$ and ${
  p}=(\cos\alpha,\sin\alpha)$. There exists a unique function
$\Psi_p\in H^{1 ,p}_{\rm loc}(\R^2,\C)$ such that 
\begin{equation}\label{eq:14}
(i\nabla +A_p)^2\Psi_p=0\quad\text{ in $\R^2$ in a
  weak $H^{1 ,p}$-sense},
\end{equation}
and 
\begin{equation}\label{eq:15}
\int_{\R^2\setminus D_r} \big|(i\nabla + A_p)(\Psi_p -
e^{\frac i2(\theta_p-\theta_0^p)} e^{\frac i2
  \theta_0}\psi )\big|^2\,dx < +\infty,
\quad\text{for any }r>1,
\end{equation}
where $\psi$ is defined  in \eqref{eq:2}.
Furthermore (see \cite[Theorem 1.5]{FFT})
\begin{equation*}
\Psi_{p} -
e^{\frac i2(\theta_{p}-\theta_0^{p})} e^{\frac i2
  \theta_0}\psi=O(|x|^{-1/2}),\quad\text{as }|x|\to+\infty.
\end{equation*}
\end{proposition}

\begin{theorem}[\cite{AF2}, Theorem 11 and Remark 12]\label{t:blowup}
  For $\alpha\in[0,2\pi)$, $p=(\cos\alpha,\sin\alpha)$ and
  $a=|a|p\in\Omega$, let $\varphi_a\in
  H^{1,a}_{0}(\Omega,\C)$ solve (\ref{eq:equation_a}-\ref{eq:6}) and
  $\varphi_0\in H^{1,0}_{0}(\Omega,\C)$ be a solution to
  \eqref{eq:equation_lambda0} satisfying \eqref{eq:1},
  \eqref{eq:131}, and \eqref{eq:54}.
 Let $\Psi_p$
  be as in Proposition \ref{prop_Psi}.  Then
\begin{equation*}
 \frac{\varphi_a(|a|x)}{|a|^{k/2}}\to \beta_2 \Psi_p \quad\text{as
  }a=|a|p\to0,
\end{equation*}
in $H^{1 ,p}(D_R,\C)$ for every $R>1$ and
in $C^{2}_{\rm loc}(\R^2\setminus\{p\},\C)$.
\end{theorem}
In the sequel, we will denote 
\[
\tilde\varphi_a(x)=\frac{\varphi_a(|a|x)}{|a|^{k/2}}.
\]
Sharp estimates of  the energy variation under moving of
poles will be derived by approximating the eigenfunction $\varphi_a$
by $H^{1 ,0}$-functions in the less expensive way from the energetic
point of view. For
every $R>2$ and $|a|$ sufficiently small, 
 we define these approximating functions
$v_{R,a}$ as follows:
\begin{equation*}
  v_{R,a}= 
 \begin{cases}
  v_{R,a}^{ext}, &\text{in }\Omega \setminus D_{R|a|},\\
  v_{R,a}^{int}, &\text{in } D_{R|a|},
 \end{cases}
\end{equation*}
where 
\begin{equation*}
  v_{R,a}^{ext} := e^{\frac{i}{2}(\theta_0^a - \theta_a)}
  \varphi_a\quad \text{in }
  \Omega \setminus D_{R|a|}
\end{equation*}
solves
\begin{equation*}
 \begin{cases}
   (i\nabla +A_0)^2 v_{R,a}^{ext} = \lambda_a v_{R,a}^{ext}, &\text{in }\Omega \setminus D_{R|a|},\\
   v_{R,a}^{ext} = e^{\frac{i}{2}(\theta_0^a - \theta_a)} \varphi_a
   &\text{on }\partial (\Omega \setminus D_{R|a|}),
 \end{cases}
\end{equation*}
whereas $v_{R,a}^{int}$ is the unique solution to the problem 
\begin{equation*}
 \begin{cases}
  (i\nabla +A_0)^2 v_{R,a}^{int} = 0, &\text{in }D_{R|a|},\\
  v_{R,a}^{int} = e^{\frac{i}{2}(\theta_0^a - \theta_a)} \varphi_a, &\text{on }\partial D_{R|a|}.
 \end{cases}
\end{equation*}
We notice that $v_{R,a}\in H^{1,0}_0(\Omega,\C)$ for all $R>2$ and $a$
sufficiently small.
For all  $R> 2$ and  $a=|a|p\in\Omega$ with
  $|a|$ small, we define 
\begin{equation}\label{eq:zar}
Z_a^R(x):=\frac{v_{R,a}^{int} (|a|x)}{|a|^{k/2}}.
\end{equation}
For all $R>2$ and $p=(\cos\alpha,\sin\alpha)$, we also define $z_{p,R}$ as the unique
solution to 
\begin{equation}\label{eq:zr}
 \begin{cases}
  (i\nabla +A_0)^2 z_{p,R} = 0, &\text{in }D_{R},\\
  z_{p,R} = e^{\frac{i}{2}(\theta_{0}^p-\theta_p)}\Psi_p, &\text{on }\partial D_{R},
 \end{cases}
\end{equation}
with  $\Psi_p$
  as in Proposition \ref{prop_Psi}.

\begin{lemma}[\cite{AF2}, Remark 12;  \cite{AF}, Lemma 8.3]\label{l:blowZ}
  For $R>2$, $\alpha\in[0,2\pi)$, $p=(\cos\alpha,\sin\alpha)$ and
  $a=|a|p\in\Omega$ small, let $\varphi_a\in
  H^{1,a}_{0}(\Omega,\C)$ solve (\ref{eq:equation_a}-\ref{eq:6}),
  $\varphi_0\in H^{1,0}_{0}(\Omega,\C)$ be a solution to
  \eqref{eq:equation_lambda0} satisfying \eqref{eq:1},
  \eqref{eq:131}, and \eqref{eq:54}, and $Z_a^R$ be as in \eqref{eq:zar}.
Then 
\[
Z_a^R\to \beta_2 z_{p,R}\quad\text{as $a=|a|p\to0$ in }H^{1 ,0}(D_R,\C) \text{ for every $R>2$},
\]
with $z_{p,R}$ being as in \eqref{eq:zr}.
\end{lemma}

\section{Variational characterization of the limit profile $\Psi_p$}\label{sec:vari-char-limit}

In \cite{AF}, the limit profile $\Psi_p$ was constructed by solving a
minimization problem in the case $p=(1,0)$ (i.e. for poles moving
tangentially to a nodal line of the limit eigenfunction); in that case
the limit profile was null
 on a half-line. In the spirit of \cite{AFNN} (where poles moving
 towards the boundary were considered), we extend this
 variational construction for poles moving along a generic direction $p=(\cos\alpha,\sin\alpha)$
 and construct  the limit profile by solving an elliptic crack problem
 prescribing the jump of the solution along the segment joining $0$
 and $p$.

Let us fix
$\alpha\in\big(0,2\pi\big)$ and
$p=(\cos\alpha,\sin\alpha)\in {\mathbb S}^1$.
We
denote by $\Gamma_p$ the segment joining $0$ to $p$, that is to say
\begin{equation*}
\Gamma_p=\{(r\cos \alpha, r\sin \alpha): r \in
(0,1)\}.
\end{equation*}
Let $s_0=\{(x_1,0):\ x_1\geq 0\}$.
We introduce the trace operators 
\begin{equation*}
 \gamma^{\pm} : \bigcap_{R>0}H^1(D_R^\pm\setminus\Gamma_p)
 \longrightarrow H^{1/2}_{\rm loc}(s_0).
\end{equation*}
We also define $\mathcal H$ as the completion of 
\[
\mathcal D= \left\{ u\in H^1(\R^2\setminus s_0)  : 
 \gamma^+(u)+\gamma^-(u)=0 \text{ on }s_0 \text{ and } 
 u=0 \text{ in neighborhoods of $0$ and }\infty \right\}
\]
with respect to the Dirichlet norm $\big(\int_{\R^2\setminus
  s_0}|\nabla u|^2\big)^{1/2}$.
In the following lemma we prove that  a
Hardy-type inequality can be recovered even in dimension $2$, under the jump condition
$\gamma^+(u)+\gamma^-(u)=0$ forced for $\mathcal H$-functions. 
\begin{lemma}\label{l:hardyD}
  The functions in $\mathcal{D}$ satisfy the following Hardy-type
  inequality:
\begin{equation*}
\int_{\R^2\setminus s_0} |\nabla \varphi(x)|^2\,dx
\geq \frac14 \int_{\R^2} \dfrac{|\varphi (x)|^2}{|x|^2}\,dx
\quad\text{for all }u\in \mathcal{D}.
\end{equation*}
\end{lemma}
\begin{proof}
 This is a consequence of a suitable change of gauge combined with the
 Hardy-type inequality for magnetic Sobolev spaces proved in \cite{LW99}.
 For any $\varphi \in \mathcal D$, the function $u:=e^{\frac{i}2\theta_0}\varphi \in \mathcal D^{1,2}_0(\R^2,\C)$ according 
 to the definition of the spaces $\mathcal D^{1,2}_p(\R^2,\C)$ given
 in Section \ref{sec:preliminaries} (see \eqref{eq:D12p}).
 From the Hardy-type inequality proved in \cite{LW99}, it follows that 
\[
 \int_{\R^2} |(i\nabla+A_0) u(x)|^2\,dx
\geq \frac14 \int_{\R^2} \dfrac{|u(x)|^2}{|x|^2}\,dx.
\]
Since $\nabla\big(\frac{\theta_0}2\big)=A_0$ and
$(i\nabla+A_0)u=ie^{\frac{i}2\theta_0}\nabla\varphi$ in
$\R^2\setminus s_0$,
we have that 
\[
 \int_{\R^2} |(i\nabla+A_0) u(x)|^2\,dx = \int_{\R^2\setminus s_0} |\nabla \varphi(x)|^2\,dx
 \quad \text{and} \quad 
 \int_{\R^2} \dfrac{|u(x)|^2}{|x|^2}\,dx = \int_{\R^2} \dfrac{|\varphi(x)|^2}{|x|^2}\,dx,
\]
thus the proof is complete.
\end{proof}
As a direct consequence of Lemma \ref{l:hardyD}, $\mathcal{H}$ can be
characterized as 
\[
\mathcal{H}=\Big\{ u\in L^1_{\rm loc}(\R^2):
 \nabla_{\R^2\setminus s_0} u\in L^2(\R^2), 
\ \tfrac{u}{|x|}\in L^2(\R^2), \text{ and } \gamma^+(u)+\gamma^-(u)=0 \text{ on }s_0\Big\},
\]
where $\nabla_{\R^2_+\setminus s_0}u$ denotes the distributional gradient of $u$ in $\R^2\setminus s_0$.

For $p\neq e$ with $e=(1,0)$, we also define the space $\mathcal{H}_p$ as the completion of 
\[
\mathcal D_p= \left\{ u\in H^1(\R^2\setminus(s_0\cup\Gamma_p))  : 
 \gamma^+(u)+\gamma^-(u)=0 \text{ on }s_0 \text{ and } 
 u=0 \text{ in  neighborhoods of $0$ and }\infty \right\}
\]
with respect to the Dirichlet norm 
\begin{equation}\label{eq:norm_Hp}
\|u\|_{\mathcal H_p}:=\|
\nabla u \|_{L^2(\R^2\setminus(s_0\cup\Gamma_p))}.
\end{equation} 
In order to prove that the space $\mathcal H_p$ defined above is
a concrete functional space, 
the argument performed in Lemma \ref{l:hardyD} is no more suitable,
since $\mathcal H_p$-functions do not satisfy a Hardy inequality in
the whole $\R^2$. We need the following two lemmas, which establish a
Hardy inequality in external domains and a Poincar\'e inequality in
$D_1$ for $\mathcal H_p$-functions.
\begin{lemma}\label{l:hardyHp}
 The functions in $\mathcal{H}_p$ 
satisfy the following Hardy-type inequality in $\R^2\setminus D_1$:
\begin{equation*}
 \| \varphi \|_{\mathcal{H}_p}^2
\geq \frac14 \int_{\R^2\setminus D_1} \dfrac{|\varphi (x)|^2}{|x|^2}\,dx,
\quad\text{for all }\varphi\in \mathcal{H}_p.
\end{equation*}
\end{lemma}
\begin{proof}
 The proof follows via a change of gauge as in the proof of Lemma
 \ref{l:hardyD}. More precisely, we notice that, for any $\varphi \in
 \mathcal D_p$, the function $u$ defined as
 $u=e^{\frac{i}2\theta_0}\varphi$ in $\R^2\setminus D_1$ and as
 $u(x)=u(x/|x|^2)$ in $D_1$ belongs to $\mathcal
 D^{1,2}_0(\R^2,\C)$. From the invariance of Dirichlet magnetic  norms
 and Hardy norms
 by Kelvin trasform and  the Hardy-type inequality of \cite{LW99}, it follows that 
\begin{align*}
\| \varphi \|_{\mathcal{H}_p}^2
&\geq\int_{\R^2\setminus (D_1\cup
  s_0)} |\nabla \varphi(x)|^2\,dx=\frac12 \int_{\R^2} |(i\nabla+A_0)
u(x)|^2\,dx \\
&\geq \frac18
 \int_{\R^2} \dfrac{|u(x)|^2}{|x|^2}\,dx = \frac14 \int_{\R^2\setminus
   D_1} \dfrac{|\varphi(x)|^2}{|x|^2}\,dx.
\end{align*}
The conclusion follows by density of $\mathcal D_p$ in $\mathcal H_p$.
\end{proof}
\begin{lemma}\label{l:poincareD1}
 The functions in $\mathcal{H}_p$ 
satisfy the following Poincar\'{e} inequality in $D_1$:
\begin{equation*}
 \| \varphi \|_{\mathcal{H}_p}^2
\geq \frac16 \int_{D_1} |\varphi(x)|^2\,dx,
\quad\text{for all }\varphi\in \mathcal{H}_p.
\end{equation*}
\end{lemma}
\begin{proof}
   From the Divergence Theorem, the Schwarz inequality and the
  diamagnetic inequality, it follows that, for every $u\in H^{1,0}(D_1\setminus\Gamma_p)$,
\begin{align*}
  2\int_{D_1} \abs{u}^2\,dx
  &=\int_{D_1\setminus \Gamma_p}\Big(\Div(|u|^2x)-2|u|\nabla|u|\cdot x\Big)\,dx\\
  &\leq\int_{\partial D_1} \abs{u}^2 \,ds
   +\int_{D_1\setminus \Gamma_p}|u|^2\,dx
+\int_{D_1\setminus \Gamma_p}|\nabla|u||^2\,dx\\
  &\leq  \int_{\partial D_1} \abs{u}^2 \,ds+
\int_{D_1} |u|^2\,dx
+
\int_{D_1\setminus \Gamma_p}
  |(i\nabla +A_0) u|^2\,dx
\end{align*}
where, when applying the Divergence Theorem, we have use the fact that
$x\cdot\nu=0$ on both sides of $\Gamma_p$. If $\varphi\in \mathcal
D_p$, then $u:=e^{\frac{i}2\theta_0}\varphi\in
H^{1,0}(D_1\setminus\Gamma_p)$ and
$(i\nabla+A_0)u=ie^{\frac{i}2\theta_0}\nabla\varphi$ in $D_1\setminus(s_0\cup\Gamma_p)$, hence
 the previous inequality yields
\[
 \int_{D_1} \abs{\varphi}^2\,dx\leq  \int_{\partial D_1} \abs{\varphi}^2 \,ds  
 + \int_{D_1\setminus(s_0\cup\Gamma_p)} |\nabla \varphi|^2\,dx.
\]
On the other hand, via the Divergence Theorem,
\begin{align*}
  &\int_{\partial D_1} \abs{\varphi}^2 
  =  \int_{\partial D_1} \varphi^2 \frac{x}{|x|^2}\cdot \nu \\
&  = - \int_{\R^2\setminus(D_1\cup s_0)}\Div\left(\varphi^2 \frac{x}{|x|^2}\right)
  + \int_0^{+\infty} \gamma^+(\varphi^2) \frac{(s,0)}{s^2}\cdot (0,-1)\,ds 
 + \int_0^{+\infty} \gamma^-(\varphi^2) \frac{(s,0)}{s^2}\cdot (0,1)\,ds 
\\
  &= - \int_{\R^2\setminus(D_1\cup s_0)}\Div\left(\varphi^2 \frac{x}{|x|^2}\right)
  = -2 \int_{\R^2\setminus(D_1\cup s_0)} \varphi\nabla \varphi \cdot \frac{x}{|x|^2}\\
  &\leq \int_{\R^2\setminus(D_1\cup s_0)} |\nabla \varphi|^2 + \int_{\R^2\setminus D_1}\frac{|\varphi|^2}{|x|^2}
  \leq 5 \|\varphi\|_{\mathcal H_p}^2,
\end{align*}
where the last inequality is obtained by Lemma \ref{l:hardyHp}.
The proof is thus complete.
\end{proof}
As a a straightforward consequence of Lemma \ref{l:hardyHp} and Lemma
\ref{l:poincareD1}, we can characterize
the space $\mathcal H_p$ as 
 \[
\Big\{ u\in L^1_{\rm loc}(\R^2):
 \nabla_{\R^2\setminus(s_0\cup\Gamma_p)} u\in L^2(\R^2), 
\ \tfrac{u}{|x|}\in L^2(\R^2\setminus D_1),
\ u\in L^2(D_1),\  \gamma^+(u)+\gamma^-(u)=0 \text{ on }s_0\Big\}.
\]
The functions in $\mathcal{H}_p$ may clearly be discontinuous on $\Gamma_p$. 
For this reason, we introduce two trace operators. 
Let us consider the sets $U^+_p=\{(x_1,x_2)\in \R^2:\cos\alpha\,
x_2>\sin\alpha\, x_1\}\cap (D_1\setminus s_0)$ and 
$U^-_p=\{(x_1,x_2)\in \R^2:\cos\alpha\,
x_2<\sin\alpha\, x_1\}\cap (D_1\setminus s_0)$.
First, for any function $u$ defined in a neighborhood of $U_p^+$,
respectively $U_p^-$, we define the restriction
\begin{equation*}
\mathcal{R}_p^+ (u) = u|_{U^+_p}, \quad \text{respectively} \quad
\mathcal{R}^-_p(u) = u|_{U^-_p}.
\end{equation*} 
We observe that, since $\mathcal{R}_p^\pm$ maps $\mathcal{H}_p$ into
$H^1(U^\pm_p)$ continuously,
the trace operators 
\[
  \gamma_p^{\pm} : \ \mathcal{H}_p \longrightarrow H^{1/2}(\Gamma_p) ,
  \quad
  u \longmapsto \gamma_p^{\pm}(u) := \mathcal{R}^{\pm}_p(u)|_{\Gamma_p}
\]
are well defined and  continuous from $\mathcal{H}_p$ to $H^{1/2}(\Gamma_p)$.
Furthermore, by  Sobolev trace
inequalities and the Poincar\'{e} inequality of Lemma \ref{l:poincareD1}, it is easy to verify that the operator norm of  $\gamma_p^{\pm}$ is bounded uniformly
with respect to $p\in{\mathbb S}^1$, in the sense that there exists
a constant $L>0$ independent of $p$ such that, recalling \eqref{eq:norm_Hp},
\begin{equation}\label{eq:11}
  \|\gamma_p^{\pm} (u)\|_{H^{1/2}(\Gamma_p)}\leq
  L\|u\|_{\mathcal{H}_p}\quad\text{for all }u\in \mathcal{H}_p.
\end{equation} 
Clearly, for a continuous function $u$, $\gamma_p^+(u) = \gamma_p^-(u)$.

Furthermore, let $\nu^+=(0,-1)$ and $\nu^-=(0,1)$ be the normal unit vectors to $s_0$, whereas 
\[
\nu^+_p = (\sin \alpha, -\cos \alpha) \quad\text{and}\quad \nu^-_p = - \nu_p^+
\]
be the normal unit vectors to $\Gamma_p$.

For every $u\in C^1(D_1 \setminus (\Gamma_p\cup s_0))$ with 
$\mathcal{R}_p^+(u)\in C^1(\overline{U_p^+}\setminus s_0)$
and $\mathcal{R}_p^-(u)\in C^1(\overline{U_p^-}\setminus s_0)$, we
define the normal derivatives $\frac{\partial^\pm u}{\partial
  \nu_p^\pm}$ on $\Gamma_p$ respectively as 
\[
\frac{\partial^+ u}{\partial \nu_p^+} := \nabla \mathcal{R}_p^+(u)
\cdot \nu_p^+\bigg|_{\Gamma_p}, \quad \text{ and } \quad
\frac{\partial^- u}{\partial \nu_p^-} := \nabla \mathcal{R}_p^-(u)
\cdot \nu_p^-\bigg|_{\Gamma_p}.
\]
Analogous definitions hold for normal derivatives on $s_0$ (which will
be denoted just as $\frac{\partial^\pm u}{\partial
  \nu^\pm}$). 

For $p\neq e$, where $e=(1,0)$, we consider the minimization problem for the functional 
$J_{p}\!:\! \mathcal H_p\! \to\R$ defined as
\begin{align}
\notag J_{p}(u) &=
  \frac{1}{2} \int_{\R^2\setminus(s_0\cup \Gamma_p)} |\nabla u|^2
\,dx + 
\int_{ \Gamma_p}  \dfrac{\partial^+ \psi}{\partial\nu_p^+}
\gamma_p^+(u)
\,ds + 
\int_{ \Gamma_p}  \dfrac{\partial^- \psi}{\partial\nu_p^-}
\gamma_p^-(u)
\,ds\\
 \label{eq:Jbis}
&=
  \frac{1}{2} \int_{\R^2\setminus(s_0\cup \Gamma_p)} |\nabla u|^2
\,dx + 
\int_{ \Gamma_p}  \dfrac{\partial^+ \psi}{\partial\nu_p^+} (\gamma_p^+(u)-\gamma_p^-(u)) \,ds
\end{align}
on the set  
\begin{equation*}
\mathcal K_p :=\{u\in \mathcal H_p : \ \gamma_p^+(u + \psi) + \gamma_p^-(u + \psi) = 0 \}.
\end{equation*}
The set $\mathcal K_p$ is nonempty, convex and closed, 
the functional $J_p$ is coercive (see \eqref{eq:J_p_coercive}), so that the problem admits a unique minimum
$w_p\in \mathcal K_p$ which is a weak solution to the problem
\begin{equation}\label{eq:wp}
 \begin{cases}
  -\Delta w_p =0, &\text{in }\R^2\setminus\{s_0\cup\Gamma_p\},\\
  \gamma^+(w_p)+\gamma^-(w_p)=0, &\text{on }s_0,\\
  \gamma^+_p(w_p+\psi)+\gamma^-_p(w_p+\psi)=0, &\text{on }\Gamma_p,\\
  \dfrac{\partial^+ w_p }{\partial\nu^+} = \dfrac{\partial^- w_p }{\partial\nu^-}, &\text{on }s_0,\\
  \dfrac{\partial^+ (w_p+\psi) }{\partial\nu_p^+} = \dfrac{\partial^-
    (w_p+\psi) }{\partial\nu_p^-},
 &\text{on }\Gamma_p.
 \end{cases}
\end{equation}

\begin{remark}\label{r:1}
 We note that the trivial function is not a solution to the problem
 \eqref{eq:wp}, since the two jump conditions for the solution and its
 normal derivative on $\Gamma_p$ cannot be satisfied simultaneously by
 the trivial function if $p\neq e$. Hence $w_p\not\equiv0$ for all $p\neq e$.
\end{remark}
One can easily see that the function $e^{\frac i2(\theta_p-\theta_0^p)e^{\frac i2\theta_0}} (w_p+\psi)$
satisfies \eqref{eq:14} and \eqref{eq:15}, hence by  the uniqueness
stated in Proposition \ref{prop_Psi} we conclude that necessarily 
\begin{equation}\label{eq:20}
\Psi_p= e^{\frac i2(\theta_p-\theta_0^p)e^{\frac i2\theta_0}} (w_p+\psi).
\end{equation}
On the other hand, for $p=e$, we consider the function $w_k\in
\Di_s(\R^2_+)$  defined as the unique minimizer in \eqref{eq:12}. The
function $w_e$ defined as 
\begin{equation}\label{eq:17}
w_e(x_1,x_2)=
\begin{cases}
  w_k(x_1,x_2),&\text{if }x_2\geq 0,\\
  w_k(x_1,-x_2),&\text{if }x_2\leq 0,
\end{cases}
\end{equation}
satisfies 
\[
w_e\in \mathcal H_e
\]
and 
\begin{equation}\label{eq:25}
 \begin{cases}
  -\Delta (w_e+\psi) =0, &\text{in }\R^2\setminus s,\\
  \gamma^+(w_e)+\gamma^-(w_e)=0, &\text{on }s,\\
  \dfrac{\partial^+ w_e }{\partial\nu^+} = \dfrac{\partial^- w_e }{\partial\nu^-}, &\text{on }s,
 \end{cases}
\end{equation}
where $s=\{(x_1,0):\ x_1\geq 1\}$ and $\mathcal H_e$ is defined as
the completion of 
\[
\mathcal D_e= \left\{ u\in H^1(\R^2\setminus s)  : 
 \gamma^+(u)+\gamma^-(u)=0 \text{ on }s \text{ and } 
 u=0 \text{ in  neighborhoods of $0$ and }\infty \right\}
\]
with respect to the Dirichlet norm 
$\|\nabla u \|_{L^2(\R^2\setminus s)}$.
One can easily see that the function $e^{\frac i2\theta_e} (w_e+\psi)$
satisfies \eqref{eq:14} and \eqref{eq:15} wit $p=e$ (notice that $\theta_0^e=\theta_0$), hence by  the uniqueness
stated in Proposition \ref{prop_Psi} we conclude that necessarily 
\begin{equation}\label{eq:21}
\Psi_e= e^{\frac i2\theta_e} (w_e+\psi).
\end{equation}
In \cite[Proposition 14]{AF2} it was proved that 
\[
\lim_{a=|a|p\to 0}\frac{\lambda_0-\lambda_a}{|a|^k}
=|\beta_2|^2k\int_0^{2\pi}w_p(\cos t,\sin t)\sin\bigg(\frac k2
t\bigg)\,dt,
\]
which, combined with \eqref{eq:4}, yields
\begin{equation}\label{eq:16}
-4{\mathfrak m}_k
\cos(k\alpha)=k\int_0^{2\pi}w_p(\cos t,\sin t)\sin\bigg(\frac k2
t\bigg)\,dt.
\end{equation}
The right hand side of \eqref{eq:16} can be related to
$J_p(w_p)$ as follows.
\begin{lemma}\label{l:3.5}
  For every $p\neq e$
\[
\int_0^{2\pi}w_p(\cos t,\sin t)\sin\bigg(\frac k2
t\bigg)\,dt=-\frac2k J_p(w_p).
\]
\end{lemma}
\begin{proof}
 Throughout this proof, let us denote 
 \[
  \omega_p(r):= \int_0^{2\pi}w_p(r\cos t,r\sin t)\sin\bigg(\frac k2 t\bigg)\,dt.
 \]
Then we have to prove that $k\omega_p(1)= -2J_p(w_p)$.
Since 
$-\Delta w_p =0$ in $\R^2\setminus\{s_0\cup\Gamma_p\}$, 
$\gamma^+(w_p)+\gamma^-(w_p)=0$ on $s_0$, and 
$\frac{\partial^+ w_p }{\partial\nu^+}=\frac{\partial^- w_p }{\partial\nu^-}$ on $s_0$, 
by direct calculations $\omega_{p}$ satisfies
\begin{equation*}
-(r^{1+k}(r^{-k/2}\omega_{p}(r))')'=0, \quad \text{in } (1,+\infty).
\end{equation*}
Hence there exists a constant $C\in\R$ such that
\[
r^{-k/2}\omega_{p}(r)=\omega_{p}(1)+\frac{C}{k}\left(1-\frac{1}{r^{k}}\right),\quad\text{for
all }r\geq 1.
\]
From \eqref{eq:20} and Proposition \ref{prop_Psi}, it follows that $\omega_{p}(r)=O(r^{-1/2})$ as $r\to+\infty$.
Hence, letting $r\to+\infty$ in the previous relation, we find $C=-k \omega_{p}(1)$, so that
$\omega_{p}(r)=\omega_{p}(1) r^{-k/2}$ for all $r\geq1$. 
By taking the derivative in this relation and in the definition of $\omega_{p}$, we obtain
\begin{equation*}
-\frac k2 \omega_{p}(1)=\int_{\partial D_1} \frac{\partial w_{p}}{\partial\nu} \psi \,ds.
\end{equation*}
Multiplying equation \eqref{eq:wp} by $\psi$ and integrating by parts
over $D_1\setminus \{s_0\cup\Gamma_p\}$, we obtain
\begin{align}
\notag\int_{D_1\setminus\{s_0\cup\Gamma_p\}}  \nabla w_{p}\cdot \nabla \psi \,dx & =
\int_{\partial D_1} \frac{\partial w_{p}}{\partial \nu} \psi \,ds
+\int_{\Gamma_p} \left( \frac{\partial^+ w_{p}}{\partial \nu_p^+} 
+ \frac{\partial^- w_{p}}{\partial \nu_p^-} \right) \psi \,ds \\
\label{eq:wpj_psij_by_parts1}& =-\frac k2 \omega_{p}(1)+\int_{\Gamma_p} \left( \frac{\partial^+ w_{p}}{\partial \nu_p^+} 
+ \frac{\partial^- w_{p}}{\partial \nu_p^-} \right) \psi \,ds.
\end{align}
Testing the equation $-\Delta\psi=0$ by $w_p$ and
integrating by parts in $D_1\setminus \{s_0\cup\Gamma_p\}$, we arrive at
\begin{align}
\notag\int_{D_1\setminus\{s_0\cup\Gamma_p\}}  \nabla w_{p}\cdot \nabla \psi \,dx &=
\int_{\partial D_1}  \frac{\partial \psi}{\partial\nu} w_{p}\,ds
+\int_{\Gamma_p} \frac{\partial^+\psi}{\partial\nu_p^+}(\gamma_p^+(w_{p})-\gamma_p^-(w_{p}))\,ds \\
  \label{eq:wpj_psij_by_parts2}&=\frac k2 \omega_{p}(1)
                                 +\int_{\Gamma_p} 
                                 \frac{\partial^+\psi}{\partial\nu_p^+}
                                 (\gamma_p^+(w_{p})-\gamma_p^-(w_{p}))\,ds,
\end{align}
where in the last step we used the fact that
$\frac{\partial \psi}{\partial\nu}=\frac k2\psi$ on $\partial D_1$.
Combining \eqref{eq:wpj_psij_by_parts1} and
\eqref{eq:wpj_psij_by_parts2}, we obtain
\begin{equation}\label{eq:omega_pj2}
k\omega_{p}(1)=\int_{\Gamma_p} \left( \frac{\partial^+ w_{p}}{\partial \nu_p^+} 
+ \frac{\partial^- w_{p}}{\partial \nu_p^-} \right) \psi \,ds
- \int_{\Gamma_p} \frac{\partial^+\psi}{\partial\nu_p^+}(\gamma_p^+(w_{p})-\gamma_p^-(w_{p}))\,ds.
\end{equation}
On the other hand, multiplying \eqref{eq:wp} by $w_p$ and integrating by parts over 
$\R^2\setminus \{s_0\cup\Gamma_p\}$, we obtain 
\begin{equation*}
\int_{\R^2\setminus\{s_0\cup\Gamma_p\}}  |\nabla w_{p}|^2 \,dx  =
\int_{\Gamma_p} \frac{\partial^+ w_{p}}{\partial \nu_p^+}\gamma_p^+(w_p)
\,ds
+\int_{\Gamma_p} \frac{\partial^- w_{p}}{\partial \nu_p^-}\gamma_p^-(w_p)
\,ds.
\end{equation*}
At the same time, recalling the definition of $J_p$ \eqref{eq:Jbis} 
and taking into account the latter equation we have
\begin{align*}
 &2J_p(w_p)   
  = \int_{\R^2\setminus\{s_0\cup\Gamma_p\}}  |\nabla w_{p}|^2 \,dx 
  + 2\int_{\Gamma_p} \frac{\partial^+\psi}{\partial\nu_p^+}\gamma_p^+(w_p) \,ds
  + 2 \int_{\Gamma_p} \frac{\partial^-\psi}{\partial\nu_p^-}\gamma_p^-(w_p) \,ds\\
 &= \int_{\Gamma_p} \frac{\partial^+ w_{p}}{\partial \nu_p^+}\gamma_p^+(w_p)
\,ds
+\int_{\Gamma_p} \frac{\partial^- w_{p}}{\partial \nu_p^-}\gamma_p^-(w_p)
\,ds
 + 2\int_{\Gamma_p} \frac{\partial^+\psi}{\partial\nu_p^+}\gamma_p^+(w_p) \,ds
  + 2 \int_{\Gamma_p} \frac{\partial^-\psi}{\partial\nu_p^-}\gamma_p^-(w_p) \,ds\\
 &= \int_{\Gamma_p} \frac{\partial^+ (w_p+\psi)}{\partial \nu_p^+}\gamma_p^+(w_p)
\,ds
+\int_{\Gamma_p} \frac{\partial^- (w_p+\psi)}{\partial \nu_p^-}\gamma_p^-(w_p)
\,ds\\
&\qquad + \int_{\Gamma_p} \frac{\partial^+\psi}{\partial\nu_p^+}\gamma_p^+(w_p) \,ds
  +  \int_{\Gamma_p} \frac{\partial^-\psi}{\partial\nu_p^-}\gamma_p^-(w_p) \,ds\\
 &= \int_{\Gamma_p} \frac{\partial^+ (w_p+\psi)}{\partial \nu_p^+}\gamma_p^+(w_p+\psi)
\,ds
+\int_{\Gamma_p} \frac{\partial^- (w_p+\psi)}{\partial \nu_p^-}\gamma_p^-(w_p+\psi)
\,ds\\
&\qquad  + \int_{\Gamma_p} \frac{\partial^+\psi}{\partial\nu_p^+}\gamma_p^+(w_p) \,ds
  +  \int_{\Gamma_p} \frac{\partial^-\psi}{\partial\nu_p^-}\gamma_p^-(w_p) \,ds\\
  &\qquad -\int_{\Gamma_p} \frac{\partial^+ (w_p+\psi)}{\partial \nu_p^+}\gamma_p^+(\psi)
\,ds
-\int_{\Gamma_p} \frac{\partial^- (w_p+\psi)}{\partial \nu_p^-}\gamma_p^-(\psi)
\,ds
  \end{align*}
from which the thesis follows by comparison with \eqref{eq:omega_pj2} 
recalling that in the last equivalence the first term is zero by
\eqref{eq:wp} 
and $\psi$ is regular on $\Gamma_p$. 
\end{proof}
From the fact that $w_k$  attains the minimum in \eqref{eq:12} and \eqref{eq:17}
it follows easily that 
\begin{equation}\label{eq:18}
{\mathfrak m}_k= \frac12\bigg[
 \frac{1}{2} \int_{\R^2\setminus s_0} |\nabla w_e|^2
\,dx + 
\int_{ \Gamma_e}  \dfrac{\partial^+ \psi}{\partial\nu^+}
\gamma^+(w_e)
\,ds + 
\int_{ \Gamma_e}  \dfrac{\partial^- \psi}{\partial\nu^-}
\gamma^-(w_e)
\,ds\bigg].
\end{equation}
Combining \eqref{eq:16}, Lemma \ref{l:3.5}, and \eqref{eq:18} we
conclude that, for every $p=(\cos\alpha,\sin\alpha)\in{\mathbb S}^1\setminus\{e\}$,
\begin{multline}\label{eq:19}
  \frac{1}{2} \int_{\R^2\setminus(s_0\cup \Gamma_p)} |\nabla w_p|^2
\,dx + 
\int_{ \Gamma_p}  \dfrac{\partial^+ \psi}{\partial\nu_p^+}
\gamma_p^+(w_p)
\,ds + 
\int_{ \Gamma_p}  \dfrac{\partial^- \psi}{\partial\nu_p^-}
\gamma_p^-(w_p)
\,ds\\
=\cos(k\alpha)\bigg[
 \frac{1}{2} \int_{\R^2\setminus s_0} |\nabla w_e|^2
\,dx + 
\int_{ \Gamma_e}  \dfrac{\partial^+ \psi}{\partial\nu^+}
\gamma^+(w_e)
\,ds + 
\int_{ \Gamma_e}  \dfrac{\partial^- \psi}{\partial\nu^-}
\gamma^-(w_e)
\,ds\bigg].
\end{multline}
\begin{lemma}\label{l:stimpre}
 \begin{enumerate}[\rm (i)]
\item There exists $C>0$ (independent of $p\in {\mathbb S}^1$) such that,
  for all $p\in {\mathbb S}^1$,
\begin{equation}\label{eq:22}
\int_{\R^2\setminus \Gamma_p} \big|(i\nabla + A_p)\Psi_p -
e^{\frac i2(\theta_p-\theta_0^p)} e^{\frac i2
  \theta_0}i\nabla \psi \big|^2\,dx \leq C.
\end{equation}
\item If $p_n,p\in  {\mathbb S}^1$ and $p_n\to p$ in ${\mathbb S}^1$,
  then $\Psi_{p_n}\to \Psi_{p}$ weakly in $H^1(D_R,\C)$
  for every $R>1$, a.e., and in $C^{0,\alpha}_{\rm loc}(\R^2\setminus\{p\})$.
\end{enumerate}
\end{lemma}
\begin{proof}
Let us fix $q>2$. From the continuity of the embedding
$H^{1/2}(\Gamma_p)\hookrightarrow L^q(\Gamma_p)$ and \eqref{eq:11}, we have that there
exists some ${\rm const\,}>0$ independent of $p\in{\mathbb S}^1$ such that, for all
$u\in \mathcal H_p$,
\begin{align*}
\left| \int_{\Gamma_p} \frac{\partial^{\pm} \psi}{\partial
    \nu_p^{\pm}} \gamma_p^{\pm} (u) \,ds \right|
&=\left|\frac{k}2 \cos \left( \frac{k}2 \alpha \right)
 \int_{\Gamma_p} |x|^{\frac{k}2-1}\gamma_p^{\pm} (u) \,ds \right|\\
& \leq \frac{k}2  \||x|^{\frac{k}2-1}\|_{L^{q'}(\Gamma_p)}
\|\gamma_p^{\pm} (u)\|_{L^q(\Gamma_p)}
\leq {\rm const\,}\|\gamma_p^{\pm} (u)\|_{H^{1/2}(\Gamma_p)}
\leq {\rm const\,} L \|u\|_{\mathcal H_p}
\end{align*}
and then, from the elementary inequality $ab\leq\frac{a^2}{4\eps}+\eps
b^2$, we deduce that, for every
$\varepsilon>0$, there exists a constant $C_\varepsilon>0$ (depending
on $\eps$ but independent of $p$) such that, for every $u\in\mathcal{H}_p$,
\begin{equation}\label{eq:J_p_coercive}
\left| \int_{\Gamma_p} \frac{\partial^{\pm} \psi}{\partial
    \nu_p^{\pm}} \gamma_p^{\pm} (u) \,ds \right| \leq
\varepsilon \|u\|_{\mathcal H_p}^2+ C_{\varepsilon}.
\end{equation}
From \eqref{eq:J_p_coercive} and the fact that the right hand side of
\eqref{eq:19} is bounded uniformly with respect to $p\in{\mathbb
  S}^1$, we deduce that for any $p=(\cos\alpha,\sin\alpha)\in{\mathbb S}^1$ 
\begin{equation}\label{eq:unifbound}
 \int_{\R^2\setminus (s_0\cup\Gamma_p)} |\nabla w_p|^2 \leq M
\end{equation}
for a constant $M>0$ independent of $p$.
Replacing \eqref{eq:20} (\eqref{eq:21} for $p=e$) into
\eqref{eq:unifbound} we obtain \eqref{eq:22}.

We have that \eqref{eq:22} together with 
the Hardy-type inequality of \cite{LW99} implies that
$\{\Psi_p\}_{p\in {\mathbb S}^1}$ is bounded in $H^1(D_R)$ and 
$\{A_p\Psi_p\}_{p\in {\mathbb S}^1}$ is bounded in $L^2(D_R)$ 
for every $R>1$. Hence, by a diagonal process, for every sequence  $p_n\to p$ in ${\mathbb
  S}^1$, there exist a subsequence (still denoted as $p_n$) and some
$\Psi\in H^1_{\rm loc}(\R^2)$ such that  $\Psi_{p_n}$
converges to $\Psi$ weakly in $H^1(D_R)$ and a.e. and
$A_{p_n}\Psi_{p_n}$ converges to $A_p\Psi$ weakly in  $L^2(D_R)$
for every $R>1$. In particular this implies that $\Psi\in H^{1
  ,p}_{\rm loc}(\R^2,\C)$. Passing to the limit in the equation
$(i\nabla +A_{p_n})^2\Psi_{p_n}=0$, we obtain that  $(i\nabla
+A_{p})^2\Psi=0$. Furthermore, by weak convergences
$\nabla\Psi_{p_n}\rightharpoonup \nabla\Psi$,
$A_{p_n}\Psi_{p_n}\rightharpoonup A_p\Psi$ in $L^2(D_R)$ and \eqref{eq:22}, we have
that, for every $R>1$, 
\begin{multline*}
\int_{D_R\setminus D_1} \big|(i\nabla + A_p)\Psi -
e^{\frac i2(\theta_{p}-\theta_0^{p})} e^{\frac i2
  \theta_0}i\nabla \psi \big|^2\,dx \\\leq
\liminf_{n\to\infty}\int_{D_R\setminus D_1} \big|(i\nabla + A_{p_n})\Psi_{p_n} -
e^{\frac i2(\theta_{p_n}-\theta_0^{p_n})} e^{\frac i2
  \theta_0}i\nabla \psi \big|^2\,dx \leq C
\end{multline*}
and, since $C$ is independent of $R$, $\int_{\R^2\setminus D_1} \big|(i\nabla + A_p)\Psi -
e^{\frac i2(\theta_{p}-\theta_0^{p})} e^{\frac i2
  \theta_0}i\nabla \psi \big|^2\,dx \leq C$. 
By  the uniqueness
stated in Proposition \ref{prop_Psi} we conclude that necessarily
$\Psi=\Psi_p$. Since the limit $\Psi$ depends neither on the sequence
$p_n$ nor on the subsequence, we obtain statement (ii).
The convergence in $C^{0,\alpha}_{\rm loc}(\R^2\setminus\{p\})$
follows by classical elliptic regularity theory.
\end{proof}

\begin{lemma}\label{l:convlq}
For every $p\in{\mathbb S}^1$, let $f_p:[0,1]\to\C$,
$f_p(r)=\Psi_p(rp)$. If $p_n,p\in {\mathbb S}^1$ and $p_n\to p$, then
$f_{p_n}\rightharpoonup f_p$ weakly in $L^q(0,1)$ for all $q>2$.  
\end{lemma}
\begin{proof}
  If $p_n\to p$ in ${\mathbb S}^1$, then the $C^{0,\alpha}_{\rm
    loc}(\R^2\setminus\{p\})$-convergence stated in Lemma
  \ref{l:stimpre} implies that $f_{p_n}\to f_p$ a.e. in $(0,1)$. 
Furthermore, from the continuity of the embedding
$H^{1/2}(\Gamma_p)\hookrightarrow L^q(\Gamma_p)$ and boundedness of
$\{\Psi_p\}_{p\in{\mathbb S}^1}$ in $H^1(D_1,\C)$, we have that 
\[
\|f_{p_n}\|_{L^q(0,1)}=
\bigg(\int_{\Gamma_{p_n}}|\Psi_{p_n}|^q\,ds\bigg)^{1/q}\leq 
{\rm const\,}\|\Psi_{p_n}\|_{H^{1/2}(\Gamma_{p_n})}
\leq 
{\rm const\,}\|\Psi_{p_n}\|_{H^{1}(D_1)}
\leq 
{\rm const\,}
\]
for positive ${\rm const}>0$ independent of $n$. Then, along a
subsequence, $f_{p_n}$ convergences weakly in $L^q(0,1)$ to some limit
which necessarily coincides with $f_p$ by a.e. convergence (then the
convergence holds not only along the subsequence).
\end{proof}

\begin{proposition}
For $\alpha\in\R$, let $p=(\cos\alpha,\sin\alpha)$. Let $w_p$ be the unique solution 
to problem \eqref{eq:wp} (\eqref{eq:25} if $p=e$). Then the function 
\begin{equation}\label{eq:24}
 \alpha \mapsto \frac12 \int_{\R^2\setminus(s_0\cup \Gamma_p)} |\nabla w_{(\cos\alpha,\sin\alpha)}(x)|^2\,dx 
\end{equation}
is continuous, even and periodic with period $\frac{2\pi}k$.
\end{proposition}
\begin{proof}
In view of \eqref{eq:19}, to prove the continuity of the map in
\eqref{eq:24} it is enough to show that the function 
\begin{equation*}
G:{\mathbb S}^1\to \R,\quad 
G(p)=
\begin{cases}
\int_{ \Gamma_p}  \frac{\partial^+ \psi}{\partial\nu_p^+}
\gamma_p^+(w_p)
\,ds + 
\int_{ \Gamma_p}  \frac{\partial^- \psi}{\partial\nu_p^-}
\gamma_p^-(w_p)
\,ds,&\text{if }p\neq e,\\[8pt]
\int_{ \Gamma_e}  \frac{\partial^+ \psi}{\partial\nu^+}
\gamma^+(w_e)
\,ds + 
\int_{ \Gamma_e}  \frac{\partial^- \psi}{\partial\nu^-}
\gamma^-(w_e)
\,ds,&\text{if }p=e,
\end{cases}
\end{equation*}
is continous. In view of \eqref{eq:20} and \eqref{eq:21}, $G$ can be
written also as 
\begin{equation*}
G(p)=
\begin{cases}
kie^{-\frac i2\theta_0(p)}\cos(\frac k2\theta_0(p))\int_0^1r^{\frac k2-1}f_p(r)\,dr,&\text{if }p\neq e,\\[8pt]
ki\int_0^1r^{\frac k2-1}f_e(r)\,dr,&\text{if }p=e,
\end{cases}
\end{equation*}
so that, to prove the continuity of $G$ it is enough to show that the
function $p\mapsto \int_0^1r^{\frac k2-1}f_p(r)\,dr$ is continuous on
${\mathbb S}^1$ and this follows from Lemma \ref{l:convlq} and the
fact that $r^{\frac k2-1}$ is in $L^t(0,1)$ for all $1<t<2$.

To the last part of the proof, 
following closely \cite[Lemma 15]{AF2}, we introduce
the two transformations $\mathcal R_1, \mathcal R_2$ acting on a general point 
\[
x=(x_1,x_2)=(r\cos t, r\sin t), \quad r>0,\ t\in[0,2\pi), 
\]
as 
\[
 \mathcal R_1 (x)= \mathcal R_1 (x_1,x_2) = 
{M}_k
 \begin{pmatrix}
   x_1\\
x_2
 \end{pmatrix},\quad
{M}_k=
 \begin{pmatrix}
   \cos\frac{2\pi}{k}&   -\sin\frac{2\pi}{k}\\[3pt]
   \sin\frac{2\pi}{k}&   \cos\frac{2\pi}{k}
 \end{pmatrix}
\]
i.e.
\[
\mathcal R_1 (r\cos t, r\sin t) = \Big(r\cos (t+\tfrac{2\pi}k), r\sin (t+\tfrac{2\pi}k)\Big),
\]
and 
\[
 \mathcal R_2 (x)= \mathcal R_2 (x_1,x_2) = (x_1,-x_2),
\]
i.e.
\[
\mathcal R_2 (r\cos t, r\sin t) = (r\cos (2\pi-t), r\sin (2\pi-t)).
\]
The transformation $\mathcal R_1$ is a rotation of $\frac{2\pi}k$ and
${\mathcal R}_2$ is a reflexion through the $x_1$-axis. 
We note that 
\begin{equation}\label{eq:falpha}
 \int_{\R^2\setminus (s_0\cup \Gamma_p)} |\nabla w_p|^2 
 =  \int_{\R^2\setminus \Gamma_p} \big|(i\nabla + A_p)\Psi_p -
e^{\frac i2(\theta_{p}-\theta_0^{p})} e^{\frac i2
  \theta_0}i\nabla \psi \big|^2\,dx. 
\end{equation}
 From the change of variable $x=\mathcal R_1(y)$ and
 \cite[Lemma 15, (58) and (66)]{AF2} we have that 
\begin{multline*}
 \int_{\R^2\setminus \Gamma_p} \big|(i\nabla + A_p)\Psi_p -
e^{\frac i2(\theta_{p}-\theta_0^{p})} e^{\frac i2
  \theta_0}i\nabla \psi \big|^2\,dx\\ = \int_{\R^2\setminus \Gamma_{\mathcal R_1^{-1}(
      p)}}\bigg|(i\nabla+A_{\mathcal
    R_1^{-1}(p)} ) 
\Psi_{\mathcal R_1^{-1}(p)}- e^{\frac i2\big(\theta_{\mathcal
      R_1^{-1}( p)}-\theta_0^ {\mathcal R_1^{-1}(
      p)}+\theta_0\big)}i\nabla\psi\bigg|^2dy
\end{multline*}
which, in view of \eqref{eq:falpha}, yields
\[
 \int_{\R^2\setminus (s_0\cup \Gamma_{\mathcal R_1^{-1}(p)})} |\nabla w_{\mathcal R_1^{-1}(p)}|^2 
=
 \int_{\R^2\setminus (s_0\cup \Gamma_p)} |\nabla w_p|^2 
\]
and hence $\frac{2\pi}{k}$-periodicity of the map \eqref{eq:24}. On
the other hand, from the change of variable $x=\mathcal R_2(y)$  and \cite[Lemma 15, (72)]{AF2} we have that 
\begin{multline*}
 \int_{\R^2\setminus \Gamma_p} \big|(i\nabla + A_p)\Psi_p -
e^{\frac i2(\theta_{p}-\theta_0^{p})} e^{\frac i2
  \theta_0}i\nabla \psi \big|^2\,dx\\ = \int_{\R^2\setminus
  \Gamma_{{\mathcal R_2(p)}}}
\bigg|(i\nabla+A_{\mathcal R_2(p)} ) 
\Psi_{\mathcal R_2(p)}- e^{\frac i2\big(\theta_{\mathcal R_2(p)}-\theta_0^{\mathcal R_2(p)}+\theta_0\big)}i\nabla\psi\bigg|^2dy
\end{multline*}
which, in view of \eqref{eq:falpha}, yields
\[
 \int_{\R^2\setminus (s_0\cup \Gamma_{\mathcal R_2(p)})} |\nabla w_{\mathcal R_2(p)}|^2 
=
 \int_{\R^2\setminus (s_0\cup \Gamma_p)} |\nabla w_p|^2 
\]
and hence evenness of the map \eqref{eq:24}.
\end{proof}

\section{Rate of convergence for
  eigenfunctions}\label{sec:rate-conv-eigenf}

In this section we prove a sharp estimate for the rate of convergence
of eigenfunctions. The estimate of the energy
variation  will be derived 
first inside disks with radius of order $|a|$ and later outside such
disks.

\subsection{Energy variation inside disks with radius of order $|a|$}

As a straightforward corollary of the blow-up results described in
section \ref{sec:preliminaries}, we obtain the
following result.

\begin{lemma}\label{l:energy_inside}
  Under the same assumptions as in Theorem \ref{t:blowup}, we have
  that, for all $p=(\cos\alpha,\sin\alpha)\in{\mathbb S}^1$ and  $R>2$, 
\begin{equation*}
\lim_{a=|a|p\to 0} \frac1{|a|^k}\int_{D_{R|a|}}\left|
(i\nabla+A_a)\varphi_a(x)-e^{-\frac i2(\theta_0^a-\theta_a)(x)}(i\nabla+A_0)\varphi_0(x)
\right|^2\,dx=|\beta_2|^2{\mathcal F}_p(R)
\end{equation*}
where 
\[
{\mathcal F}_p(R)=\int_{D_{R}}\left|
(i\nabla+A_p)\Psi_p(x)-e^{-\frac i2(\theta_0^p-\theta_p)(x)}(i\nabla+A_0)(e^{\frac i2\theta_0}\psi)(x)
\right|^2\,dx.
\]
\end{lemma}
\begin{proof}
By a change of variable we obtain that 
\begin{multline*}
  \int_{D_{R|a|}}\left|
(i\nabla+A_a)\varphi_a(x)-e^{-\frac i2(\theta_0^a-\theta_a)(x)}(i\nabla+A_0)\varphi_0(x)
\right|^2\,dx\\=|a|^k
 \int_{D_{R}}\left|
(i\nabla+A_p)\tilde\varphi_a(x)-e^{-\frac i2(\theta_0^p-\theta_p)(x)}(i\nabla+A_0)W_a(x)
\right|^2\,dx
\end{multline*}
so that the conclusion follows from convergence \eqref{eq:vkext_la}
and  Theorem \ref{t:blowup}.
\end{proof}

\begin{lemma}\label{l:limfR}
 Let ${\mathcal F}_p(R)$ be as in Lemma \ref{l:energy_inside}. Then
 \[
  \lim_{R\to+\infty}{\mathcal F}_p(R) = {\mathfrak L}_p>0
\]
where 
\begin{equation*}
{\mathfrak L}_p=
 \int_{\R^2\setminus (\Gamma_p\cup s_0)} |\nabla w_p|^2 
\end{equation*}
and $w_p$ is the weak solution to the problem \eqref{eq:wp}. 
\end{lemma}
\begin{proof}
 Via a change of gauge, we can write 
 \[
  {\mathcal F}_p(R)= \int_{D_R\setminus (\Gamma_p\cup s_0)} \left| e^{\frac i2(\theta_p-\theta_0^p)}e^{\frac i2 \theta_0}
  \left(i\nabla (w_p+\psi) - i\nabla\psi\right)\right|^2 
  = \int_{D_R\setminus (\Gamma_p\cup s_0)} \left| \nabla w_p\right|^2
  \rightarrow \int_{\R^2\setminus (\Gamma_p\cup s_0)} \left|
    \nabla w_p\right|^2
 \]
 as $R\to+\infty$. Thanks to remark \ref{r:1}, we stress that the limit is non zero. 
 This concludes the proof.
\end{proof}

\subsection{Energy variation outside disks with radius of order $|a|$}
 
In order to estimate the energy variation outside disks with radius
$R|a|$, we consider the following operator:
\begin{align*}
  F: \C \times H^{1,0}_{0}(\Omega,\C) &\longrightarrow
  \R \times \R \times (H^{1,0}_{0,\R}(\Omega,\C))^\star\\
  (\lambda,\varphi) &\longmapsto \Big( {\textstyle{
      \|\varphi\|_{H^{1,0}_0(\Omega,\C)}^2 -\lambda_0,\
      \mathfrak{Im}\big(\int_{\Omega}
      \varphi\overline{\varphi_0}\,dx\big), \ (i\nabla +A_0)^2
      \varphi-\lambda \varphi}}\Big).
\end{align*}
In the above definition, $(H^{1,0}_{0,\R}(\Omega,\C))^\star$ is the real dual space of
  $H^{1,0}_{0,\R}(\Omega,\C)=H^{1,0}_{0}(\Omega,\C)$, which is here meant as a
vector space over $\R$ endowed with the norm
\[
\|u\|_{H^{1,0}_0(\Omega,\C)}=\bigg(
\int_{\Omega}\big|(i\nabla +A_0)u\big|^2dx\bigg)^{\!\!1/2},
\]
and $(i\nabla +A_0)^2  \varphi-\lambda \varphi\in (H^{1,0}_{0,\R}(\Omega,\C))^\star$ acts as 
\[
\phantom{a}_{(H^{1,0}_{0,\R}(\Omega,\C))^\star}\!\Big\langle (i\nabla
+A_0)^2 \varphi-\lambda \varphi , u
\Big\rangle_{\!H^{1,0}_{0,\R}(\Omega,\C)}\!\!=\mathfrak{Re}
\left({\textstyle{\int_{\Omega}(i\nabla+A_0)\varphi\cdot\overline{(i\nabla+A_0)u}\,dx
      -\!\lambda \!\int_{\Omega} \varphi\overline{u} \,dx}}\right)
\]
for all $\varphi\in H^{1,0}_{0,\R}(\Omega,\C)$.

\begin{lemma}\label{l:energy_outside}
For $\alpha\in[0,2\pi)$, $p=(\cos\alpha,\sin\alpha)$ and
  $a=|a|p\in\Omega$, let $\varphi_a\in
  H^{1,a}_{0}(\Omega,\C)$ solve (\ref{eq:equation_a}-\ref{eq:6}) and
  $\varphi_0\in H^{1,0}_{0}(\Omega,\C)$ be a solution to
  \eqref{eq:equation_lambda0} satisfying \eqref{eq:1},
  \eqref{eq:131}, and \eqref{eq:54}. Then, for all $R>2$, 
\begin{equation*}
  \| e^{\frac{i}{2}(\theta_0^a - \theta_a)}(i\nabla+A_a)\varphi_a
  -(i\nabla+A_0)\varphi_0 \|^2_{L^2(\Omega\setminus D_{R|a|},\C)}
  \leq |a|^{k} g(a,R)
\end{equation*}
where, for all $R>2$, 
\begin{equation}\label{eq:10}
\lim_{a=|a|p\to0}g(a,R)=g(R)
\end{equation}
and 
\begin{equation}\label{eq:9}
\lim_{R\to+\infty}g(R)=0.
\end{equation}
\end{lemma}
\begin{proof}
From \cite[Lemma 7.1]{AF} we know that the function
$F$ is Fr\'{e}chet-differentiable at $(\lambda_0,\varphi_0)$ and its
  Fr\'{e}chet-differential $dF(\lambda_0,\varphi_0)$ 
is invertible.
From the invertibility of $dF(\lambda_0,\varphi_0)$ 
it follows that 
\begin{align*}
\big\| &e^{\frac{i}{2}(\theta_0^a - \theta_a)}(i\nabla+A_a)\varphi_a
-(i\nabla+A_0)\varphi_0 \big\|_{L^2(\Omega\setminus D_{R|a|},\C)}
\\
&=
\big\| (i\nabla+A_0)(e^{\frac{i}{2}(\theta_0^a - \theta_a)}
\varphi_a-\varphi_0) \big\|_{L^2(\Omega\setminus D_{R|a|},\C)}\\
&\leq
  |\lambda_{a} - \lambda_0| + \|v_{R,a} -
  \varphi_0\|_{H^{1,0}_0(\Omega,\C)}\\
 & \leq\|(dF(\lambda_0,\varphi_0))^{-1}\|_{ \mathcal L( \R\times \R
    \times (H^{1,0}_{0,\R}(\Omega,\C))^\star,\C \times
    H^{1,0}_{0}(\Omega,\C))} \| F(\lambda_a,v_{R,a})\|_{
    \R\times\R \times (H^{1,0}_{0,\R}(\Omega))^\star} (1+o(1))
\end{align*}
as $|a|\to 0^+$.
We have that 
\[
  F(\lambda_a,v_{R,a}) =\left( \alpha_a, \beta_a, w_a\right)
\]
where 
\begin{align*}
\alpha_a&=\|v_{R,a}\|_{H^{1,0}_0(\Omega,\C)}^2
    -\lambda_0\in\R,\\
 \beta_a&=
    \mathfrak{Im}\left({\textstyle{\int_{\Omega}v_{R,a}\overline{\varphi_0}\,dx}}\right)\in\R,\\
w_a&=
    (i\nabla+A_0)^2 v_{R,a} - \lambda_a v_{R,a} \in (H^{1,0}_{0,\R}(\Omega))^\star.
\end{align*}
We mention that in \cite{AF,AF2}, the norm of $\| F(\lambda_a,v_{R,a})\|_{
    \R\times\R \times (H^{1,0}_{0,\R}(\Omega))^\star}$ was estimated
  before proving the blow-up results recalled in Theorem \ref{t:blowup} and Lemma
  \ref{l:blowZ} (actually some preliminary estimates of $F(\lambda_a,v_{R,a})$ were carried
  out to obtain an energy control in terms of an implicit  normalization needed to prove the blow-up
  results). Here we are going to exploit the sharp blow-up results Theorem \ref{t:blowup} and Lemma
  \ref{l:blowZ} to improve the preliminary estimates in \cite{AF,AF2}.
From  \eqref{eq:4}, Theorem \ref{t:blowup} and Lemma
  \ref{l:blowZ}  we have that 
\begin{align*}
 \alpha_a 
 &= \left(
 \int_{ D_{R|a|}} |(i\nabla+A_0)v_{R,a}^{int}|^2 \,dx-
 \int_{D_{R|a|}} |(i\nabla+A_a)\varphi_a|^2\,dx 
 \right) +(\lambda_a-\lambda_0)\\
 &=  |a|^k
\left(
 \int_{ D_{R}} |(i\nabla+A_0)Z_a^R|^2 \,dx-
 \int_{D_{R}} |(i\nabla+A_p)\tilde\varphi_a|^2\,dx 
 \right) +(\lambda_a-\lambda_0)= O(|a|^k),
\end{align*}
as $|a|\to0^+$.
The normalization
condition for the phase in \eqref{eq:6} together with the blow-up
results  \eqref{eq:vkext_la}, Theorem \ref{t:blowup} and Lemma
  \ref{l:blowZ} yield
\begin{align*}
  \beta_a &= \Im \left( \int_{D_{R|a|}}
    v_{R,a}^{int}\overline{\varphi_0} \,dx- \int_{D_{R|a|}}
    e^{\frac i2 (\theta_0^a-\theta_a)} \varphi_a\overline{\varphi_0}\,dx
    +\int_{\Omega} e^{\frac i2 (\theta_0^a-\theta_a)}
    \varphi_a\overline{\varphi_0}\,dx
  \right)\\
  &=
 \Im \left( |a|^{k+2}\int_{D_{R}}
    Z_a^R\overline{W_a} \,dx- |a|^{k+2}\int_{D_{R}}
    e^{\frac i2 (\theta_0^p-\theta_p)} \tilde\varphi_a\overline{W_a}\,dx
    \right)=O(|a|^{k+2}) \quad \text{as }|a|\to
  0^+.
\end{align*}
Let $\mathcal D^{1,2}_0(\R^2,\C)$ be the functional space  defined in \eqref{eq:D12p}.
For every $a\in\Omega$, we  define the map
\[
\mathcal T_a:\mathcal D^{1,2}_0(\R^2,\C)\to \mathcal
D^{1,2}_0(\R^2,\C),\quad  \mathcal T_a\varphi(x)=\varphi(|a|x).
\]
 It is
easy to verify that $\mathcal T_a$ is an isometry of $\mathcal
D^{1,2}_0(\R^2,\C)$. 

Since $H^{1,0}_{0}(\Omega,\C)$ can be thought as continuously
embedded into $\mathcal D^{1,2}_0(\R^2,\C)$ by trivial extension
outside $\Omega$ and $\|u\|_{\mathcal D^{1,2}_0(\R^2,\C)}=
\|u\|_{H^{1,0}_{0}(\Omega,\C)}$ for every $u\in H^{1,0}_{0}(\Omega,\C)$, we have that 
\begin{align}\label{eq:8}
 \notag &\|w_a\|_{(H^{1,0}_{0,\R}(\Omega,\C))^\star} \\
 \notag & = \sup_{\substack{\varphi\in H^{1,0}_{0}(\Omega,\C) \\
      \|\varphi\|_{H^{1,0}_0(\Omega,\C)}=1}} \bigg|\Re
  \left(\int_{\Omega}(i\nabla+A_0)v_{R,a}\cdot\overline{(i\nabla+A_0)\varphi}\,dx
    -\lambda_a \int_{\Omega} v_{R,a}\overline{\varphi}
    \,dx\right)\bigg| \\
  & \leq \sup_{\substack{\varphi\in \mathcal D^{1,2}_0(\R^2,\C)\\
      \|\varphi\|_{\mathcal D^{1,2}_0(\R^2,\C)}=1}} \bigg|\Re
  \left(\int_{\Omega}(i\nabla+A_0)v_{R,a}\cdot\overline{(i\nabla+A_0)\varphi}\,dx
    -\lambda_a \int_{\Omega} v_{R,a}\overline{\varphi}
    \,dx\right)\bigg|.
\end{align}
For every $\varphi \in  H^{1,0}_{0}(\Omega,\C) $ we have that 
\begin{align}\label{eq:5}
 \notag &\int_{\Omega}(i\nabla+A_0)v_{R,a}\cdot\overline{(i\nabla+A_0)\varphi}\,dx
    -\lambda_a \int_{\Omega} v_{R,a}\overline{\varphi}
    \,dx\\
\notag&=\int_{\Omega\setminus D_{R|a|}}
 e^{\frac i2 (\theta_0^a-\theta_a)}  (i\nabla+A_a)\varphi_a\cdot\overline{(i\nabla+A_0)\varphi}\,dx
    -\lambda_a \int_{\Omega\setminus D_{R|a|}} e^{\frac i2 (\theta_0^a-\theta_a)} \varphi_a\overline{\varphi}
    \,dx\\
&\qquad+
\int_{D_{R|a|}}(i\nabla+A_0)v_{R,a}\cdot\overline{(i\nabla+A_0)\varphi}\,dx
    -\lambda_a \int_{D_{R|a|}} v_{R,a}\overline{\varphi}
    \,dx.
\end{align}
From scaling and integration by parts
\begin{align}\label{eq:3}
\notag  &\int_{\Omega\setminus D_{R|a|}}
 e^{\frac i2 (\theta_0^a-\theta_a)}  (i\nabla+A_a)\varphi_a\cdot\overline{(i\nabla+A_0)\varphi}\,dx
    -\lambda_a \int_{\Omega\setminus D_{R|a|}} e^{\frac i2 (\theta_0^a-\theta_a)} \varphi_a\overline{\varphi}
    \,dx\\
\notag  &=|a|^{\frac k2}\left(\int_{\frac \Omega{|a|}\setminus D_{R}}
 e^{\frac i2 (\theta_0^p-\theta_p)}
  (i\nabla+A_p)\tilde\varphi_a\cdot\overline{(i\nabla+A_0)(\mathcal T_a \varphi)}\,dx
    -\lambda_a |a|^2\int_{\frac \Omega{|a|}\setminus D_{R}} e^{\frac
  i2 (\theta_0^p-\theta_p)}
 \tilde\varphi_a\overline{\mathcal T_a\varphi}
    \,dx\right)\\
\notag  &=|a|^{\frac k2}\left(\int_{\frac \Omega{|a|}\setminus D_{R}}
   (i\nabla+A_p)\tilde\varphi_a\cdot\overline{(i\nabla+A_p)(
e^{-\frac i2 (\theta_0^p-\theta_p)}
\mathcal T_a \varphi)}\,dx
    -\lambda_a |a|^2\int_{\frac \Omega{|a|}\setminus D_{R}}
 \tilde\varphi_a\overline{e^{-\frac i2 (\theta_0^p-\theta_p)}\mathcal T_a\varphi}
    \,dx\right)\\
&\notag=|a|^{\frac k2}i\int_{\partial D_{R}}\overline{\mathcal T_a\varphi}
 e^{\frac i2 (\theta_0^p-\theta_p)} (i\nabla+A_p)\tilde\varphi_a\cdot
  \nu\,d\sigma\\
&=
|a|^{\frac k2}i\int_{\partial D_{R}}\overline{\mathcal T_a\varphi}
(i\nabla+A_0)( e^{\frac i2 (\theta_0^p-\theta_p)}\tilde\varphi_a)\cdot
  \nu\,d\sigma
\end{align}
being $\nu=\frac x{|x|}$ the outer unit vector. 
In a similar way we have that 
\begin{align}\label{eq:7}
\notag  \int_{D_{R|a|}}&(i\nabla+A_0)v_{R,a}\cdot\overline{(i\nabla+A_0)\varphi}\,dx
    -\lambda_a \int_{D_{R|a|}} v_{R,a}\overline{\varphi}
    \,dx\\
\notag  &=|a|^{\frac k2}\left(\int_{D_{R}}
   (i\nabla+A_0)Z_a^R\cdot\overline{(i\nabla+A_0)(
\mathcal T_a \varphi)}\,dx
    -\lambda_a |a|^2\int_{D_{R}}Z_a^R\overline{\mathcal T_a\varphi}
    \,dx\right)\\
&=|a|^{\frac k2}\left(-i\int_{\partial D_{R}}
   (i\nabla+A_0)Z_a^R\cdot\nu \overline{\mathcal T_a \varphi}\,d\sigma
    -\lambda_a |a|^2\int_{D_{R}}Z_a^R\overline{\mathcal T_a\varphi}
    \,dx\right).
\end{align}
Combining \eqref{eq:8}, \eqref{eq:5}, \eqref{eq:3}, \eqref{eq:7}, and
recalling that $\mathcal T_a$ is an isometry of $\mathcal D^{1,2}_0(\R^2,\C)$, we obtain
that
\begin{align*}
  &|a|^{-\frac k2}\|w_a\|_{(H^{1,0}_{0,\R}(\Omega,\C))^\star}\\
  &\quad
    \leq \sup_{\substack{\varphi\in \mathcal D^{1,2}_0(\R^2,\C)\\
  \|\varphi\|_{\mathcal D^{1,2}_0(\R^2,\C)}=1}} \bigg|i\int_{\partial D_{R}}
  (i\nabla+A_0)\left( e^{\frac i2 (\theta_0^p-\theta_p)}\tilde\varphi_a-Z_a^R\right)\cdot
  \nu\,\overline{\mathcal T_a\varphi}\,d\sigma
  -\lambda_a |a|^2\int_{D_{R}}Z_a^R\overline{\mathcal T_a\varphi}
  \,dx\bigg|\\
  &\quad=
    \sup_{\substack{\varphi\in \mathcal D^{1,2}_0(\R^2,\C)\\
  \|\varphi\|_{\mathcal D^{1,2}_0(\R^2,\C)}=1}} \bigg|i\int_{\partial D_{R}}
  (i\nabla+A_0)\left( e^{\frac i2 (\theta_0^p-\theta_p)}\tilde\varphi_a-Z_a^R\right)\cdot
  \nu\,\overline{\varphi}\,d\sigma
  -\lambda_a |a|^2\int_{D_{R}}Z_a^R\overline{\varphi}
  \,dx\bigg|\\
  &\quad
    \leq \sup_{\substack{\varphi\in \mathcal D^{1,2}_0(\R^2,\C)\\
  \|\varphi\|_{\mathcal D^{1,2}_0(\R^2,\C)}=1}} \bigg|\int_{\partial D_{R}}
  (i\nabla+A_0)\left( e^{\frac i2 (\theta_0^p-\theta_p)}\tilde\varphi_a-Z_a^R\right)\cdot
  \nu\,\overline{\varphi}\,d\sigma\bigg|
  \\
  &\qquad\qquad+\lambda_a |a|^2\sup_{\substack{\varphi\in \mathcal D^{1,2}_0(\R^2,\C)\\
  \|\varphi\|_{\mathcal D^{1,2}_0(\R^2,\C)}=1}}\bigg|
  \int_{D_{R}}Z_a^R\overline{\varphi}
  \,dx\bigg|\\
  &\quad=
    \sup_{\substack{\varphi\in \mathcal D^{1,2}_0(\R^2,\C)\\
  \|\varphi\|_{\mathcal D^{1,2}_0(\R^2,\C)}=1}} \bigg|\int_{\partial D_{R}}
  (i\nabla+A_0)\left( e^{\frac i2 (\theta_0^p-\theta_p)}\tilde\varphi_a-Z_a^R\right)\cdot
  \nu\,\overline{\varphi}\,d\sigma\bigg|
  +|a|^2O\left(\|Z_a^R\|_{L^2(D_R,\C)}\right).
\end{align*}
From Theorem \ref{t:blowup} and Lemma \ref{l:blowZ} it follows that 
\[
(i\nabla+A_0)\left( e^{\frac i2 (\theta_0^p-\theta_p)}\tilde\varphi_a-Z_a^R\right)\cdot
  \nu\to\beta_2
(i\nabla+A_0)\left( e^{\frac i2 (\theta_0^p-\theta_p)}\Psi_p-z_{p,R}\right)\cdot
  \nu\quad\text{in }H^{-1/2}(\partial D_R)
\]
as $a=|a|p\to0$
and 
\[
|a|^2O\left(\|Z_a^R\|_{L^2(D_R,\C)}\right)\to 0 \quad\text{as }a=|a|p\to0.
\]
Hence we conclude that 
\[
|a|^{-\frac k2}\|w_a\|_{(H^{1,0}_{0,\R}(\Omega,\C))^\star}\leq h(a,R)
\]
with 
\[
\lim_{a=|a|p\to0}h(a,R)=|\beta_2|h(R)
\]
being
\[
h(R)=\sup_{\substack{\varphi\in \mathcal D^{1,2}_0(\R^2,\C)\\
      \|\varphi\|_{\mathcal D^{1,2}_0(\R^2,\C)}=1}} \bigg|\int_{\partial D_{R}}
(i\nabla+A_0)\left( e^{\frac i2 (\theta_0^p-\theta_p)}\Psi_p-z_{p,R}\right)\cdot
  \nu\,\overline{\varphi}\,d\sigma\bigg|.
\]
We observe that, for every $\varphi\in \mathcal D^{1,2}_0(\R^2,\C)$,
\begin{align*}
&\bigg|\int_{\partial D_{R}}
(i\nabla+A_0)\left( e^{\frac i2 (\theta_0^p-\theta_p)}\Psi_p-z_{p,R}\right)\cdot
  \nu\,\overline{\varphi}\,d\sigma\bigg|\\
&=
\bigg|\int_{\partial D_{R}}e^{\frac i2 (\theta_0^p-\theta_p)}
(i\nabla+A_p)\left( \Psi_p- e^{\frac i2 (\theta_p-\theta_0^p)} e^{\frac i2
  \theta_0}\psi\right)\cdot
  \nu\,\overline{\varphi}\,d\sigma
\\
&\hskip5cm+ \int_{\partial D_{R}}
(i\nabla+A_0)\left(  e^{\frac i2
  \theta_0}\psi -z_{p,R}\right)\cdot
  \nu\,\overline{\varphi}\,d\sigma\bigg|\\
&=\bigg|-i\int_{\R^2\setminus D_{R}}
(i\nabla+A_p)\left( \Psi_p- e^{\frac i2 (\theta_p-\theta_0^p)} e^{\frac i2
  \theta_0}\psi\right)\cdot
\overline{(i\nabla+A_0)\varphi}e^{\frac i2
  (\theta_0^p-\theta_p)}\,dx\\
&\qquad+i
\int_{D_R}(i\nabla+A_0) \left(  e^{\frac i2
  \theta_0}\psi -z_{p,R}\right)\cdot
  \overline{(i\nabla+A_0)\varphi}\,dx\bigg|\\
&\leq \bigg(\sqrt{\int_{\R^2\setminus D_{R}}
\left|(i\nabla+A_p)\left( \Psi_p- e^{\frac i2 (\theta_p-\theta_0^p)}e^{\frac i2
  \theta_0}\psi\right)\right|^2 \,dx}\\
&\hskip4cm+
\sqrt{\int_{D_R}\left|(i\nabla+A_0) \left(  e^{\frac i2
  \theta_0}\psi -z_{p,R}\right)\right|^2\,dx}\bigg)\|\varphi\|_{\mathcal D^{1,2}_0(\R^2,\C)}
\end{align*}
and hence 
\begin{multline*}
h(R)\leq \sqrt{\int_{\R^2\setminus D_{R}}
\left|(i\nabla+A_p)\left( \Psi_p- e^{\frac i2 (\theta_p-\theta_0^p)}e^{\frac i2
  \theta_0}\psi\right)\right|^2 \,dx}\\
+
\sqrt{\int_{D_R}\left|(i\nabla+A_0) \left(  e^{\frac i2
  \theta_0}\psi -z_{p,R}\right)\right|^2\,dx}.
\end{multline*}
From Proposition \ref{prop_Psi} it follows that $\lim_{R\to+\infty}\int_{\R^2\setminus D_{R}}
\big|(i\nabla+A_p)\big( \Psi_p- e^{\frac i2 (\theta_p-\theta_0^p)}e^{\frac i2
  \theta_0}\psi\big)\big|^2 \,dx=0$.

Since $(i\nabla+A_0)^2 \big(  e^{\frac i2
  \theta_0}\psi -z_{p,R}\big)=0$ in $D_R$ and 
$(e^{\frac i2
  \theta_0}\psi -z_{p,R})\big|_{\partial D_R}=e^{\frac i2
  \theta_0}\psi -e^{\frac{i}{2}(\theta_{0}^p-\theta_p)}\Psi_p$,
if $\eta_R$ is a
smooth cut-off function satisfying 
\[
\eta_R\equiv 0 \text{ in }D_{R/2},\quad
\eta\equiv 1 \text{ in }\R^2\setminus D_R,\quad
0\leq\eta_R\leq1, \quad|\nabla \eta_R|\leq \frac 4R
\quad \text{in }D_R\setminus D_{R/2},
\]
from the Dirichlet
Principle we can estimate
\begin{align*}
&\int_{D_R}\left|(i\nabla+A_0) \left(  e^{\frac i2
  \theta_0}\psi -z_{p,R}\right)\right|^2\,dx 
\leq \int_{D_R}\left|(i\nabla+A_0) \left( \eta_R(e^{\frac i2
  \theta_0}\psi -e^{\frac{i}{2}(\theta_{0}^p-\theta_p)}\Psi_p)\right)\right|^2\,dx 
\\
&\leq 2\int_{D_R}|\nabla \eta_R|^2|  
e^{\frac i2
  \theta_0}\psi -e^{\frac{i}{2}(\theta_{0}^p-\theta_p)}\Psi_p|^2\,dx 
+2\int_{\R^2\setminus D_{R/2}}\left|(i\nabla+A_0) \left( 
e^{\frac i2
  \theta_0}\psi -e^{\frac{i}{2}(\theta_{0}^p-\theta_p)}\Psi_p\right)\right|^2\,dx 
\\
&\leq \frac{32}{R^2}\int_{D_R\setminus D_{R/2}}| \Psi_p- 
e^{\frac{i}{2}(\theta_p-\theta_{0}^p)} e^{\frac i2
  \theta_0}\psi |^2\,dx 
+2\int_{\R^2\setminus D_{R/2}}
\left|(i\nabla+A_p)\left( \Psi_p- e^{\frac i2 (\theta_p-\theta_0^p)}e^{\frac i2
  \theta_0}\psi\right)\right|^2 \,dx 
\end{align*}
which, in view of Proposition \ref{prop_Psi}, implies that
$\lim_{R\to+\infty}\int_{D_R}\big|(i\nabla+A_0) \big(  e^{\frac i2
  \theta_0}\psi -z_{p,R}\big)\big|^2\,dx=0$. Therefore we can conclude
that $h(R)\to0$ 
as $R\to+\infty$. The proof is thereby complete.
\end{proof}

\subsection{Proof of Theorem \ref{t:main}}

As observed in \S \ref{sec:preliminaries}, it is not restrictive to
assume $\beta_1=0$. 
Let $\eps>0$. From Lemma \ref{l:limfR} and \eqref{eq:9} there exists
$R_0>0$ sufficiently large such that 
\[
|{\mathcal F}_p(R_0)-{\mathfrak L}_p|<\eps\quad\text{and}\quad 
|g(R_0)|<\eps.
\]
From \eqref{eq:10} and Lemma \ref{l:energy_inside} there exists
$\delta>0$ (depending on $\eps$ and $R_0$) such that, if $|a|<\delta$, then 
\[
|g(a,R_0)-g(R_0)|<\eps
\]
and 
\[
\left|\frac1{|a|^k}\int_{D_{R_0|a|}}\left|
(i\nabla+A_a)\varphi_a(x)-e^{-\frac i2(\theta_0^a-\theta_a)(x)}(i\nabla+A_0)\varphi_0(x)
\right|^2\,dx-|\beta_2|^2{\mathcal F}_p(R_0)\right|<\eps.
\]
Therefore, taking into account also Lemma \ref{l:energy_outside},
 we have that, for all $a=|a|p$ with $|a|<\delta$, 
\begin{align*}
  \bigg||a|^{-k}&\int_\Omega \left|
 (i\nabla+A_a)\varphi_a-e^{-\frac{i}{2}(\theta_0^a-\theta_a)}(i\nabla+A_0)\varphi_0\right|^2\,dx-|\beta_2|^2{\mathfrak
  L}_p\bigg|\\
&\leq\bigg||a|^{-k}\int_{D_{R_0|a|}} \left|
 (i\nabla+A_a)\varphi_a-e^{-\frac{i}{2}(\theta_0^a-\theta_a)}(i\nabla+A_0)\varphi_0\right|^2\,dx-|\beta_2|^2 {\mathcal F}_p(R_0)\bigg|\\
&\quad +|a|^{-k}\int_{\Omega\setminus D_{R_0|a|}} \left|
 (i\nabla+A_a)\varphi_a-e^{-\frac{i}{2}(\theta_0^a-\theta_a)}(i\nabla+A_0)\varphi_0\right|^2\,dx+|\beta_2|^2|{\mathfrak
  L}_p-{\mathcal F}_p(R_0)|\\
&<\eps+g(a,R_0)+|\beta_2|^2\eps\leq
\eps+|g(a,R_0)-g(R_0)|+|g(R_0)|+|\beta_2|^2\eps =(3+|\beta_2|^2)\eps,
\end{align*}
thus concluding the proof of Theorem \ref{t:main}.

\end{document}